\documentclass[11pt]{article}
\usepackage{xypic}
\usepackage{epsfig}
\usepackage{tikz}
\usepackage{graphicx}
\usepackage{amsthm}
\usepackage{amssymb}
\usepackage{caption}
\usepackage{subcaption}
\usepackage{color}
\usepackage{hyperref}
\usepackage{indentfirst}
\usepackage{amsmath,mathrsfs,amsfonts,verbatim,enumitem,leftidx}
\usepackage{extarrows}
\usepackage{dsfont}
\usepackage{ulem}

\newtheorem{defn}{Definition}[section]
\newtheorem{cro}[defn]{Corollary}
\newtheorem{cons}[defn]{Construction}

\newtheorem{prop}[defn]{Proposition}
\newtheorem{thm}[defn]{Theorem}
\newtheorem{lem}[defn]{Lemma}
\newtheorem{rem}[defn]{\bf Remark}
\newtheorem{exmp}[defn]{Example}

\numberwithin{equation}{section}

\def\mr#1{{{\mathrm{#1}}}\setcounter{equation}{0}}
\def\mc#1{{{\mathcal{#1}}}\setcounter{equation}{0}}

\newcommand{\p}{\ensuremath{\mathbf{p}}}
\newcommand{\q}{\ensuremath{\mathbf{q}}}
\newcommand{\tube}{\ensuremath{\mathbf{t}}}

\begin{document}

\title{Minimal silting modules {and ring extensions}
\footnotetext{ Version of \today}
\footnotetext {$^{\ast}$ Corresponding author.}
\footnotetext{ E-mail addresses: lidia.angeleri@univr.it(Lidia Angeleri H\"{u}gel), caoweiqing18@163.com(Weiqing Cao)}}

\author{ Lidia Angeleri H\"{u}gel$^{a}$,  Weiqing Cao$^{b\ast}$\\
\small $^{a}$Universit\`a degli Studi di Verona, Strada le Grazie 15 - Ca' Vignal, I-37134 Verona, Italy\\
  \small    $^{b}$School of Mathematical Sciences, Nanjing Normal University,
    Nanjing 210023, Jiangsu, P.R.China\\
}

\date{}
\maketitle
{\small {{\bf Abstract.}}
{\small Ring epimorphisms often induce silting modules and cosilting modules, termed minimal silting or minimal cosilting. The aim of this paper is twofold. Firstly, we determine the minimal tilting and minimal cotilting modules over  a tame hereditary algebra. In particular, we show that a large cotilting module is minimal if and only if it has an adic module as a direct summand. Secondly,  we discuss the behaviour of minimality under ring extensions. We show that minimal cosilting modules over a commutative noetherian ring extend to minimal cosilting modules along any flat ring epimorphism. Similar results are obtained for  commutative rings of small homological dimension.}


\smallskip

{\small {\bf Keywords. } Minimal silting modules. Ring epimorphisms. Ring extensions. Minimal cosilting modules. Tame hereditary algebras.






{\section{Introduction}}

Tilting theory and its recent development into silting theory are known to be closely related to localization of rings. For example, every Ore localization $R\hookrightarrow \Sigma^{-1}R$ of a ring $R$ with the property that $\Sigma^{-1}R$ has projective dimension at most one over $R$ gives rise to a tilting module
$\Sigma^{-1}R\oplus \Sigma^{-1}R/R$.
Of course, such tilting module will often be large, that is, it won't be finitely presented,  not even up to equivalence.

More generally, ring epimorphisms with nice homological properties give rise to silting modules. Such modules were introduced in \cite{AMV1} as  large analogues of the support $\tau$-tilting modules studied in representation theory and cluster theory. They can be characterized as zero cohomologies of (not necessarily compact) two-term silting complexes.

Building on these connections, it was shown in \cite{AMV2} that the universal localizations of a hereditary ring are parametrized  by certain silting modules which are determined by a minimality condition and are called {\it minimal silting}. A dual version of this result was recently established in \cite{AH}, leading to the notion of a {\it minimal cosilting}  module. {The interest in minimal silting or cosilting modules goes well beyond the hereditary case. For example,  minimal cosilting modules also parametrize the flat ring epimorphisms starting in a commutative noetherian ring}. We refer to Section 2 for details.

In the present paper, we continue these investigations by analyzing two aspects.
The first aspect concerns an important and widely studied class of hereditary rings:  finite dimensional tame hereditary algebras. The large cotilting modules over such algebras were classified in \cite{Buan}. They are determined by their indecomposable summands, which can be either finite dimensional regular modules, or infinite dimensional pure-injective, thus Pr\"ufer modules, adic modules, or the generic module.  A classification of the large tilting modules was established in \cite{AS2}. Both tilting and cotilting modules are parametrized by pairs $(Y,P)$ where $Y$ is a branch module, that is, a finite dimensional regular module with certain combinatorial features, and $P$ is a subset of the projective line, when the ground field is algebraically closed, or more generally, a subset of the index set $\mathbb X$ of the tubular family $\tube=\bigcup_{\lambda\in\mathbb{X}}\tube_\lambda$ in the Auslander-Reiten quiver.

In Section 3, we determine the  minimal tilting and minimal cotilting modules over  a tame hereditary algebra. Since the finite dimensional (co)tilting modules are all minimal, we focus on the large ones.  We prove that under the parametrization described above, minimal tilting or minimal cotilting modules correspond to the pairs $(Y,P)$ where $P$ is not empty. This result (Theorem~\ref{main}) is achieved by an explicit construction of the universal localization corresponding to the tilting module $T_{(Y,P)}$ 
when $P\neq\emptyset$. More precisely, we   construct the wide subcategory $\mc{M}$ of the category of finite dimensional modules that allows to realize $T_{(Y,P)}$ as the tilting module $R_{\mc{M}}\oplus R_{\mc{M}}/R$ arising from the universal localization  $R\to R_{\mc{M}}$  of $R$ at  $\mc{M}$.
We also obtain  that a large cotilting module is minimal if and only if it has an adic direct summand (Theorem~\ref{main2}).

The second aspect we want to address is the behaviour of silting and cosilting modules under ring extensions. A criterion recently established in \cite{BREAZ} ensures that every silting module $T$ over a commutative ring $R$ extends to a silting module $T\otimes_RS$  along any ring epimorphism $R\to S$. In Section 4, we give conditions under which  minimality is preserved. In particular, we show that all minimal cosilting modules over a commutative noetherian ring extend to minimal cosilting modules along any flat ring epimorphism (Corollary~\ref{cor1}). Over a commutative hereditary ring, we see that every cosilting module  extends to a minimal cosilting module along any ring epimorphism (Corollary~\ref{cor2}).

The paper is organized as follows.
 Section 2 contains some preliminaries on ring epimorphisms and a survey on their relation with silting and cosilting modules.
In Section 3, we determine the minimal tilting and minimal cotilting modules over a tame hereditary algebra. Section 4 is devoted to extensions of minimal silting (or cosilting) modules along ring epimorphisms. We first study the example of the Kronecker algebra (Example~\ref{fail}).
Then we turn to commutative rings and provide  some useful criteria for preserving minimality. We close the paper with applications to commutative noetherian rings and commutative rings of small homological dimension.

\section{Preliminaries}
\subsection{Notation}
Throughout the paper, $R$ will denote a ring, Mod$R$ ($R\,$Mod) the category of all right (left) $R$-modules and  mod$R$ ($R\,$mod) the category of finitely presented right (left) $R$-modules.

We fix a   commutative ring $k$  such that  $R$ is a $k$-algebra, together with   an injective cogenerator $W$ in $\mr{Mod} \,k$,  and we denote by   $(-)^+ = \mr{Hom}_{k}({-},{W})$  the duality functors between $\mr{Mod}{R}$ and ${R}\mr{Mod}$. For example, one can choose $k=\mathbb Z$ and  $W= \mathbb{Q}/\mathbb{Z}$. In case $R$ is a finite dimensional algebra over a field $k$, we will take the usual vector space duality $(-)^+= D= \mr{Hom}_{k}({-},{k})$.

 Let $\mc{C}\subset$ Mod$R$ be a class of modules. Denote
by Add$\mc{C}$ (respectively, add$\mc{C}$) the class consisting of all modules isomorphic to direct summands of (finite) direct sums of elements of M. The class consisting of all modules isomorphic
to direct summands of products of modules of $\mc{C}$ is denoted by Prod$\mc{C}$. The class consisting
of the right $R$-modules which are epimorphic images of arbitrary direct sums of elements in $\mc{C}$
is denoted by Gen$\mc{C}$. Dually, we define Cogen$\mc{C}$ as the class of all submodules of arbitrary
direct products of elements in $\mc{C}$. Moreover, we write
\begin{verse}
 ${\mc{C}^{\perp}}\ =\{N_{R}\ |$
  $\mr{Ext}^{1}_{R}(M,N)=0=\mr{Hom}_{R}(M,N)$ for each
$M\in\mc{C}\}$.

 ${\mc{C}^{\perp_{1}}}\ =\{N_{R}\ |$
  $\mr{Ext}_{R}^{1}(M,N)=0$ for each
$M\in\mc{C}\}$.

 ${\mc{C}^{\perp_0}}\ =\{N_{R}\ |$
  $\mr{Hom}_{R}(M,N)=0$ for each
$M\in\mc{C}\}$.
\end{verse}
and define  $^{\perp}\mc{C}$, $^{\perp_{1}}\mc{C}$, $^{\perp_0}\mc{C}$ dually. If $\mc{C}$ contains a unique module $M$, then we shall denote these subcategories by $M^{\perp}$, $M^{\perp_{1}}$, and $M^{\perp_0}$ etc.

\subsection{Ring epimorphisms}


\begin{defn}{\rm
A ring homomorphism $\lambda: R\rightarrow S$ is a {\it ring epimorphism}  if it is an epimorphism in the category of rings with unit, or equivalently, if the functor given by restriction of scalars \mbox{$\lambda_{\ast}: \mathrm{Mod}S\rightarrow \mathrm{Mod}R$} is a full embedding.

\smallskip

 A ring epimorphism $\lambda: R\rightarrow S$ is said to be
\begin{itemize}
\item {\it homological}  if $\mathrm{Tor}^{R}_{i}(S,S)=0$ for $i> 0$, or equivalently, the functor given by restriction of scalars $\lambda_{\ast}:D(\mr{Mod}S)\rightarrow D(\mr{Mod}R) $ induces a full embedding of the corresponding derived categories;
\item   {\it (right) flat}  if $S$ is a flat right $R$-module;
\item  {\it pseudoflat}  if $\mathrm{Tor}^{R}_{1}(S,S)=0$.
\end{itemize}

Two ring epimorphisms $\lambda:  R\rightarrow S$ and $\lambda^{'}:  R\rightarrow S^{'}$ are {\it equivalent} if there is a ring isomorphism $h: S\rightarrow S^{'}$ such that $\lambda^{'}=h\cdot \lambda$. We say that $\lambda$ and $\lambda'$ lie in the same {\it epiclass} of $R$.
}
\end{defn}

Ring epimorphisms are closely related to certain subcategories of Mod$R$.


\begin{defn}{\rm
A full subcategory $\mathcal{X}$ of Mod$R$ is called {\it bireflective} if the inclusion functor $\mathcal{X}\rightarrow \mathrm{Mod}R$ admits both a left and right adjoint, or equivalently,  $\mc{X}$  is closed under products, coproducts, kernels and cokernels.
}
\end{defn}


\begin{thm}\label{22}{\rm \cite{Gabriel,BD}}
$(1)$ The map assigning to a ring epimorphism  $\lambda:R\rightarrow S$ the essential image $\mc{X}$ of the functor $\lambda_\ast$ defines a bijection between

~~~~~~~$(i)$ epiclasses of ring epimorphisms $R\rightarrow S$;

~~~~~~$(ii)$ bireflective subcategories $\mc{X}$ of $\mathrm{Mod}R$.

$(2)$ The following statements are equivalent for a ring epimorphism $\lambda: R \rightarrow S$.

~~~~~~$(1)$ $\lambda$ is a pseudoflat ring epimorphism;

~~~~~~$(2)$ $\mathcal{X}$ is closed under extensions in $\mathrm{Mod}R$;

~~~~~~$(3)$ the functors $\mathrm{Ext}^{1}_{R}$ and $\mathrm{Ext}^{1}_{S}$ agree on $S$-modules;

~~~~~~$(4)$ the functors $\mathrm{Tor}_{1}^{R}$ and $\mathrm{Tor}_{1}^{S}$ agree on $S$-modules.
\end{thm}

Classical localization of commutative rings at multiplicative sets provides an important class of examples for flat ring epimorphisms. More generally, the notion of universal localization which we recall below
yields a large supply of pseudoflat ring epimorphisms. If $R$ is a hereditary ring, then
 $\lambda: R \rightarrow S$ is a homological ring epimorphism if and only if it is pseudoflat, which is equivalent   to being a universal localization of $R$ by {\cite [Theorem ~6.1] {KRAUSE}}.

\begin{thm}{\rm\cite [Theorem 4.1] {Scho}}
Let $R$ be a ring and $\Sigma$ be a class of morphisms between finitely generated projective right $R$-modules. Then there is {a} pseudoflat ring epimorphism $\lambda: R \rightarrow R_{\Sigma}$, called  {\rm universal localization} of $R$ at $\Sigma$, such that

$(1)$ $\lambda$ is $\Sigma$-inverting: if $\sigma$ belongs to $\Sigma$, then $\sigma\otimes_{R}R_{\Sigma}$ is an isomorphism of right $R_{\Sigma}$-modules, and

$(2)$ $\lambda$ is universal $\Sigma$-inverting: for any $\Sigma$-inverting morphism $\lambda^{'}: R\rightarrow S$ there exists a unique ring homomorphism $g: R_{\Sigma}\rightarrow S$ such that $g\circ \lambda=\lambda^{'}$.
\end{thm}

Given a bireflective subcategory {$\mc{X}$} of Mod$R$ and an $R$-module $M$, we will denote by \mbox{$\psi_{M}: M\rightarrow X_{M}$} the unit of the
adjunction given by the left adjoint of the inclusion functor. The map $\psi_{M}$ is an $\mc{X}$-{\it reflection}, i.e.~Hom$_{R}(\psi_{M}, X)$ is an isomorphism
for all $X$ in $\mc{X}$. In particular, $\psi_{M}$   is a left $\mc{X}$-approximation which is {\it left minimal},
i.e.~any endomorphism $\theta$ of ${X}_{M}$ with $\theta\circ \psi_{M}=\psi_{M}$ is an isomorphism.

\subsection{Silting theory}




 Given  a    morphism $\sigma: P\to Q$ between  projective modules, we define the subcategory
$$\mathcal{D}_{\sigma}=\{X\in \text{Mod}R|~ \text{Hom}_{R}(\sigma,X) ~\text{is~surjective}\}. $$

\begin{defn}{\rm \cite{AMV1,AMV4}
We say that  an ${R}$-module $T$
\begin{itemize}
\item admits a {\it presilting presentation} if there is a projective presentation  $P\stackrel{\sigma}\rightarrow Q\rightarrow T\rightarrow 0$ such that Hom$_{D(\mr{Mod}R)}(\sigma,\sigma^{(I)}[1])=0$ for all sets $I$ or, equivalently,  $\mathrm{Gen}{T}\subseteq D_{\sigma}$;
\item is a {\it silting} module  if it admits a projective presentation $P\stackrel{\sigma}\rightarrow Q\rightarrow T\rightarrow 0$ with $ \mathrm{Gen}T=\mathcal{D}_{\sigma}$, in
which case we say that $T$ is silting {\it with respect to} $\sigma$;
\item  is a {\it tilting} module if it is silting with respect to an injective map $\sigma$, or equivalently, $\mr{Gen}T=T^{\perp_{1}}$.
\\ This amounts to  the following conditions:

$(T1)$ proj.dim(T)$\leq 1$,

$(T2)$ $\mr{Ext^{1}_{R}}(T,T^{(\kappa)})=0$ for any cardinal $\kappa${,}

$(T3)$ there is an exact sequence $ 0\rightarrow R\rightarrow T_{0}\rightarrow T_{1}\rightarrow 0$ with $T_{0}$, $T_{1}\in \mr{Add}T$.
\end{itemize}
}
\end{defn}
\noindent
Note that every silting module $T$ satisfies Add$T=\mr{Gen}T\cap{^{\perp_{1}}(\mr{Gen}T)}$. Moreover, $T$ gives rise to a torsion pair with torsion class $\mr{Gen}T$ and torsion-free class $T^{\perp_0}$. The class $\mr{Gen}T$ is called a {\it silting class}, or a  {\it tilting class}  in case $T$ is a tilting module.  Silting modules having the same silting class are said to be {\it equivalent}. We say that a  silting module is {\it large} if it not equivalent to a finitely presented silting module.
{\it Cosilting} or {\it cotilting} modules and  classes are  defined dually, and equivalence of  cosilting or cotilting modules is defined correspondingly.

\smallskip

If $T$ is a silting module with respect to $\sigma$, then by \cite{WEI} there is a triangle
$$ R\stackrel{\phi}\rightarrow \sigma_{0} \rightarrow \sigma_{1} \rightarrow R[1]~~~~~~~~~~~~~~~~~~~~~~~~~~~~(1.1)$$
in the derived category $D(\mr{Mod}R)$, where $\sigma_{0}$ and $\sigma_{1}$ lie in $\text{Add}\sigma$ and $\phi$ is a left $\text{Add}\sigma$-approximation of $R$. Applying the cohomology functor $\text{H}^{0}(-)$ to this triangle, we obtain an exact sequence
$$R\stackrel{f}\rightarrow T_{0} \rightarrow T_{1} \rightarrow0 ~~~~~~~~~~~~~~~~~~~~~~~~~~~~~~(1.2)$$
in Mod$R$, where $T_{0}, T_{1}\in \text{Add}{T}$ and $f$ is a left $\text{Add}{T}$-approximation of $R$.

\begin{defn}{\rm \cite{AMV4}
Let $T$ be a silting module with respect to $\sigma$. If the map $\phi$ in the triangle $(1.1)$ can be chosen left minimal, then $T$ is said to be a {\it minimal silting module}.}
\end
{defn}

For example, all  finite dimensional silting (that is, support $\tau$-tilting) modules over a finite dimensional algebra are minimal, cf.~\cite[Remark 1.6]{AMV4}.
Minimal silting modules are closely related with pseudoflat ring epimorphisms. Indeed, there is a map assigning a pseudoflat ring epimorphism to every minimal silting module  \cite[Corollary 2.4]{AMV4}. Conversely, every ring epimorphism $\lambda:R\to S$ for which the right $R$-module $S_R$ admits a presilting presentation induces a silting  $R$-module of the form $T=S\oplus \mr{Coker}\lambda$, see \cite[Proposition 1.3]{AMV4}.

\begin{defn}{\rm  We say that a silting module $T$ {\it arises from a ring epimorphism} if there is a ring epimorphism $\lambda:R\to S$ such {that} $S\oplus \mr{Coker}\lambda$ is a silting $R$-module equivalent to $T$.}
\end{defn}

\begin{prop}\label{bijpresilting}
{ The map assigning to a ring epimorphism $\lambda:R\to S$ the right $R$-module $S\oplus\mr{Coker} \lambda$  yields  an injection from (i) to (ii), where
\begin{enumerate}
\item[(i)] epiclasses of ring epimorphisms $\lambda:R\to S$ such that $S_R$ admits a presilting presentation,
\item[(ii)] equivalence classes of silting right $R$-modules arising from ring epimorphisms.
\end{enumerate}
If $R$ is  right perfect, then this map is a bijection, and all modules in (ii) are minimal silting modules.}
\end{prop}
\begin{proof}
If  a silting module $T=S\oplus \mr{Coker}\lambda$ arises from a ring epimorphism $\lambda:R\to S$, then the map $\lambda$, viewed as an $R$-module homomorphism, is a minimal left  $\mr{Add}T$-approximation of $R$ and is thus uniquely determined up to isomorphism. This shows that the equivalence class of $T$ determines the bireflective subcategory $\mc{X}= \{X\in \mr{Mod}R\,\mid\, \mr{Hom}_R(\lambda,X) \text{ is bijective}\}$ and therefore the epiclass of $\lambda$, proving  the injectivity of the assignment.

Now, if $R$ is right perfect, then every silting module $T=S\oplus \mr{Coker}\lambda$ in (ii) is minimal by \cite[Remark 1.6]{AMV4}, and its presilting presentation entails the existence of a presilting presentation for the direct summand $S_R$ (compare the minimal projective presentations of $T_R$ and $S_R$). Thus the assignment is also surjective.
\end{proof}



The bijection for right perfect rings in Proposition~\ref{bijpresilting} has  a dual version, which holds over arbitrary rings thanks to the existence of minimal  injective copresentations. Let us first introduce the necessary terminology.
If $C$ is a cosilting  left $R$-module, then by \cite{WZ} there is an  exact sequence
\begin{equation}\label{precover} 0 \rightarrow C_{1}\rightarrow C_{0}\stackrel{g}\rightarrow R^{+}\end{equation}
where $C_0,C_1$ are in $\mathrm{Prod}{C}$, and $g$ is a right $\mathrm{Prod}{C}$-approximation of $_RR^+$.
\begin{defn}{\rm\cite{AH}}
{\rm
(1)
A cosilting left $R$-module $C$ is a  {\it minimal cosilting module} if the exact sequence (\ref{precover})
can be chosen such that
the subcategory $\mathrm{Cogen}{C}\cap {^{\perp_{0}}C_{1}}$ is bireflective,
and $\mathrm{Hom}_{R}(C_{0},C_{1})=0$.

\smallskip

\noindent
(2) We say that a module $C$ admits a {\it precosilting copresentation} if there is an injective copresentation $0\rightarrow C\rightarrow Q_{0}\stackrel{\omega}\rightarrow Q_{1}$ such that Hom$_{D(A)}(\omega^{I},\omega[1])=0$ for all sets $I$.
}
\end{defn}

\begin{thm}\label{mincos}{\rm \cite[Theorem 4.17]{AH}}
 {The map assigning to a ring epimorphism $\lambda:R\to S$  the  left $R$-module $S^{+}\oplus \mr{Ker}\lambda^{+}$, yields a bijection between
 \begin{enumerate}
\item[(i)] epiclasses of  ring epimorphisms $\lambda:R\to S$ such that $_RS^+$ has a precosilting copresentation,
\item[(ii)] equivalence classes of minimal cosilting left $R$-modules.
\end{enumerate}}
\end{thm}

Also minimal cosilting modules are intimately related with pseudoflat  ring epimorphisms.
\begin{rem}\label{pf} {\rm \cite[Example 4.15]{AH},\cite[Proposition 4.5]{AMSTV} The ring epimorphisms satisfying condition (i) above are all pseudoflat, and the converse holds true if $S_R$ has weak dimension at most one. If $R$ is commutative noetherian, then (i) consists precisely of the epiclasses of  flat ring epimorphisms.}\end{rem}

In the hereditary case we obtain the following result.

\begin{thm}\label{00}{\rm \cite[Theorem 5.8 and Corollary 5.17]{AMV2}, \cite[Corollary~4.22]{AH}}.
If $R$ is hereditary,  there are bijections between
\begin{enumerate}
\item[(i)] epiclasses of homological ring epimorphisms $R\to S$,
\item[(ii)] equivalence classes of minimal silting right $R$-modules,
\item[(iii)]  equivalence classes of minimal cosilting  left $R$-modules,
\item[(iv)] subcategories of $\mr{mod}R$ which are {wide}, i.e.~closed under kernels, cokernels, and extensions.
\end{enumerate}
The bijections $(i)\to (ii), (iii)$  map a homological ring epimorphism $\lambda:R\to S$ to the silting right $R$-module  $S\oplus\mr{Coker} \lambda$ and to the cosilting  left $R$-module $S^{+}\oplus \mr{Ker}\lambda^{+}$, and restrict to bijections between injective homological ring epimorphisms, tilting right modules and cotilting left modules.
The assignment $(ii)\to (iii)$  is given by  $T \mapsto T^+$. The bijection
  $(iv)\to (i)$ maps a wide subcategory $\mc{M}$ to the universal localization $R\rightarrow R_{\mc{M}}$ at  (projective resolutions of) the modules in $\mc{M}$.

 \end{thm}

\section{Minimal tilting  modules over tame hereditary algebras}

In this section, let $R$ be a finite dimensional tame hereditary  algebra over a field $k$, which we assume to be indecomposable. We want to determine the  minimal tilting modules over $R$. Since every finite dimensional tilting module is obviously minimal, we will focus on the large tilting modules.

\subsection{Preliminaries on tame hereditary algebras}
It is well known that the finite dimensional indecomposable (right) $R$-modules are depicted by the Auslander-Reiten quiver of
mod$R$, which consists
of a preprojective and a preinjective component, denoted by  $\p$ and $\q$, respectively, and    a  family of orthogonal tubes
$\tube=\bigcup_{\lambda\in\mathbb{X}}\tube_\lambda$ containing the  regular modules. For details we refer e.g.~to \cite{ARS}.

Given a quasi-simple  (or simple regular) module $S$, that is, a module at the mouth of  a tube $\tube_\lambda$, we denote by $S[m]$ the module of regular length $m$ on the ray
$$S=S[1]\subset S[2]\subset\cdots\subset S[m]\subset S[m+1]\subset \cdots~$$
and let
$S[\infty]=\lim\limits_{m\to\infty}S[m]$ be the corresponding {\it Pr\"{u}fer} module. The {\it adic} module $S[-\infty]$ corresponding to $S$ is defined as the inverse limit along the coray ending at $S$.
We denote by $G$ the {\it generic} module. It is the unique indecomposable infinite dimensional module which has finite length over its endomorphism ring.





{\subsection{{Over the Kronecker algebra}}\label{Kronecker}}

We start by reviewing the case when $R$ is the Kronecker algebra, i.e.~the path algebra of the quiver
$ \bullet\rlap{\raise3pt\hbox{$\xrightarrow{}$}}
        \lower4pt\hbox{$\xrightarrow[]{}$}\,\bullet$.
 Denote by $P_{i}$ (respectively $Q_{i}$), with $i\in\mathbb{N}$,
 the (finite dimensional) indecomposable preprojective (respectively,
preinjective) right $R$-modules, indexed such that dim$_{k}$Hom$_{R}(P_{i},P_{i+1})=2$ (respectively, dim$_{k}$Hom$_{R}(Q_{i+1}, Q_{i})=2$). Recall that $P_{1}$ is simple projective and
embeds in all Kronecker modules but the modules in Add$Q_{1}$, and $Q_{1}$ is simple
injective with a surjection from all Kronecker modules but the modules in Add$P_{1}$.

By \cite[Theorem~6.1]{KRAUSE}, every homological ring epimorphism $R\to S$ is equivalent to a universal localization at a set of (projective resolutions of) finitely presented modules. Here is a complete list of the epiclasses of $R$, together with the corresponding bireflective subcategories $\mc{X}$ of Mod$R$:

$\bullet $ $R\rightarrow 0$ and $id_{R}: R \rightarrow R$,

$\bullet$ the universal localization at $P_{1}$ with $\mc{X}=\mr{Add}Q_{1}$,

 $\bullet$ the universal localization at $P_{i+1}, i\geq 1$, with $\mc{X}=\mr{Add}P_{i}$,

$\bullet$ the universal localization at $Q_{i},i\geq 1$, with $\mc{X}=\mr{Add}Q_{i+1}$,

$\bullet$ the universal localization at a non-empty set $\mc{U}$ of simple regular modules, with $\mc{X}=\mc{U}^{\perp}$.

\smallskip

\noindent
Notice that the epimorphisms in this list are either surjective with an idempotent
kernel, or injective, and the only non-injective ones are $R\rightarrow 0$ and the universal
localizations at the projective modules $P_{1}$ and $P_{2}$.

\smallskip

The following is a complete list of silting
right $R$-modules, up to equivalence:

$\bullet$ $0,P_{1},Q_{1}$, the only silting modules that are not tilting,

$\bullet$ $P_{i}\oplus P_{i+1}, i\geq 1$,

$\bullet$ $Q_{i+1} \oplus Q_{i}, i\geq 1,$

$\bullet$ $R_{\mc{U}}\oplus R_{\mc{U}}/R$, where  $\mc{U}$ is a non-empty set of simple regular modules,

$\bullet$  the {Lukas tilting module} {$\mathbf{L}$ with  Gen$\mathbf{L}={^{\perp_0}\mathbf{p}}$, the unique non-minimal silting module.

\smallskip

\noindent
For details we refer to \cite{AS2}, \cite[Example 5.19]{AMV2}.

\medskip
\subsection{Classification of tilting modules}

Let us now return to an arbitrary tame hereditary algebra $R$. The large  tilting $R$-modules have been classified in \cite{AS2}. In contrast to the Kronecker case, they can  have finite dimensional summands. This is due to the existence of (at most three) non-homogeneous tubes.   Notice, however, that the finite dimensional part of a large tilting module can be described explicitly. In order to explain this, we need to recall some terminology.

\begin{defn} {\rm
(1) An $R$-module $Y$ is said to be {\it exceptional} if $\mr{Ext}^{1}_{R}(Y,Y)=0$.

(2) Given a tube $\mathbf{t}_{\lambda}$ of rank $r> 1$ and a module $X=U[m]\in\mathbf{t}_{\lambda}$ of regular length $m< r$, we consider the full subquiver $\mc{W}_{X}$ of $\mathbf{t}_{\lambda}$ which is isomorphic to the Auslander-Reiten-quiver $\Theta(m)$ of the linearly oriented quiver of type $\mathbb{A}_{m} $ with $X$ corresponding to the projective-injective vertex of $\Theta(m)$.  The set $\mc{W}_{X}$   is called a {\it wing} of $\mathbf{t}_{\lambda}$ of size $m$, which is {\it rooted in the  vertex} $X=U[m]$.

(3) A finite dimensional regular multiplicity-free exceptional $R$-module $Y$ is a {\it branch module} if it satisfies the following condition:
For each quasi-simple module $S$ and $m\in\mathbb{N}$ such that $S[m]$ is a direct summand of $Y$, there exist precisely $m$ direct summands of $Y$ that belong to $\mc{W}_{S[m]}$.
}
\end{defn}

In other words, a branch module is a regular multiplicity-free exceptional module whose indecomposable summands are arranged in disjoint wings, and the number of summands from each wing equals the size of that wing.

\smallskip

Now it is shown in \cite{AS2} that the large tilting modules over $R$
are parametrized by pairs $(Y,P)$ where $Y$ is a branch module, and $P$ is subset of $\mathbb{X}$.
More precisely, every such  pair $(Y,P)$ determines two sets of quasi-simple modules:
{
\begin{itemize}\item
the set $\mc{V}=\mc{V}_{(Y,P)} $ given by all quasi-simple modules in $\bigcup\limits_{\lambda\in P}\mathbf{t}_{\lambda}$ and all regular composition factors of $Y$,
\item
 the set
$\mc{U}=\mc{U}_{(Y,P)} $ given by all quasi-simple modules in $\bigcup\limits_{\lambda\in P}\mathbf{t}_{\lambda}$ that are not regular composition factors of $\tau^-Y$.
\end{itemize}
}
\noindent With these sets one can construct a tilting module
$$T_{(Y,P)}=Y\oplus(\mathbf{L}\otimes_{R}R_{\mc{V}})\oplus\bigoplus_{S\in \mc{U}} S[\infty],$$
and it turns out that the  modules $T_{(Y,P)}$ form
 a complete irredundant list of all large tilting modules, up to equivalence. For details we refer to \cite{AS2,A}.


Our aim is to show that a large tilting module $T_{(Y,P)}$ is minimal if and only if the set $P\subset \mathbb{X}$ is non-empty. The only-if part of this statement is already contained in a result from \cite{AS2}, which we briefly recall  for the reader's convenience.

\begin{prop}\label{101}{\rm (cf.~\cite[Corollary ~5.10]{AS2})} If $T_{(Y,P)}$ is a minimal tilting module, then the set $P$ is not empty.
\end{prop}
\begin{proof} We just outline the argument and refer to \cite{AS2} for details.  Assume  $T=T_{(Y,P)}$ is equivalent to a tilting module of the form $S\oplus S/R$ for some ring epimorphism $R\to S$. If $P=\emptyset$, then $T=
Y\oplus(\mathbf{L}\otimes_{R}R_{\mc{V}})$ where $\mc{V}$ is the set of regular composition factors of $Y$. Now one considers the torsion and torsion-free part of $T$ with respect to the torsion pair  $(\mr{Gen}\tube, \tube^{\perp_0})$ generated by $\tube$. It turns out that $Y$ is torsion and $\mathbf{L}\otimes_{R}R_{\mc{V}}$ is torsion-free. Moreover,   $\mr{Add}Y$ contains all torsion modules in $\mr{Add}T$, and  $\mr{Add}(\mathbf{L}\otimes_{R}R_{\mc{V}})$ contains all torsion-free modules in $\mr{Add}T$.
Next, one shows that $S/R$ is torsion and therefore lies in $\mr{Add}Y$. One then deduces that $(S/R)^{\perp_1}=Y^{\perp_1}$. Notice that $(S/R)^{\perp_1}=\mr{Gen}T$, hence our tilting class $\mr{Gen}T$ coincides with  $Y^{\perp_1}$. On the other hand, by a well-known result due to Bongartz, the finite dimensional exceptional module $Y$ can be completed to a finite dimensional tilting module with tilting class $Y^{\perp_1}$. But this contradicts the assumption that $T$ is large.
\end{proof}

\subsection{Minimal tilting modules}

Let us fix a large tilting module $T=T_{(Y,P)}$  with $P\neq\emptyset~$.
We want to prove that $T$ is minimal. To this end, we will use Theorem \ref{00} and show that $T_{(Y,P)}$ arises from universal localization at a wide subcategory $\mc{M}$ of mod$R$.

We start out with an easy, but useful observation.
\begin{rem}\label{perp} {\rm (cf.~\cite[Example~4.4]{AS2}) If $S$ is a quasi-simple module, then $S[\infty]^{\perp_{1}}=\bigcap\limits_{n\geq 1}S[n]^{\perp_{1}}$.

More generally, let $\mc{E}$ be a class of modules, and suppose that a module $X$ lies in $\mc{E}^{\perp_1}$. Then $X$ lies in ${E}^{\perp_1}$ for every module $E$ which is filtered by modules from $\mc{E}$, or which is a submodule of some module in $\mc{E}$.

This follows immediately from the fact that any class of the form ${}{^{\perp_1}}X$ is closed under filtrations by \cite[3.1.2]{gobel}, and when $X$ has injective dimension at most one (as in our case), then it is also closed under submodules.
}\end{rem}

Next, we compute the tilting class of $T$. This amounts to computing
the subcategory $$\mc{S}={^{\perp_{1}}\mr{Gen}T}\cap\mr{mod}R,$$ that is, the largest subcategory $\mc{S}$ of mod$R$ with the property that $\mr{Gen}T=\mc{S}^{\perp_{1}}.$ We are going to see that the indecomposable non-preprojective modules in $\mc{S}$ either lie on a ray starting in $\mc{U}=\mc{U}_{(Y,P)}$, or lie ``below'' an indecomposable summand of $Y$. Here is the precise statement.

\begin{lem}\label{7}

$(1)$ A finitely generated  indecomposable $R$-module $M$ belongs to  $\mc{S}$ if and only if one of the following statements holds true:\begin{itemize}
\item[-] $M$ is  preprojective,
\item[-] $M\cong S[n]$ where $n\geq 1$ and $S$ is in $\mc{U}$,
\item[-]  there is a module $S[h]\in\mr{Add}Y$ such that $M\cong S[i]$ for some   $1\leq i\leq h$.
\end{itemize}
 $(2)$ $\mr{Gen}T=\bigcap\limits_{S\in\mc{U}}S[\infty]^{\perp_{1}}\cap Y^{\perp_{1}}$.

\end{lem}

\begin{proof} $(1)$ We know from \cite[Theorem 2.7]{AS2} that $\mc{S}=\mr{add}(\mathbf{{p}}\cup \mathbf{t}^{'})$, where $\mathbf{t}^{'}$ is a set of regular modules. Take a quasi-simple $S$ whose ray $\{S[n]\,|\, n\geq 1\}$ contains some modules from $\mathbf{t}^{'}$. If the whole ray is contained in $\mathbf{t}^{'}$, then by \cite[Theorem 4.5]{AS2}, the Pr\"ufer module $S[\infty]\in\mr{Add}T$, and $S\in\mc{U}$. If $\tube'$ contains some, but not all modules from the ray, then $\mathbf{t}^{'}\cap\{S[n]\,\mid\, n\geq 1\}=\{S[i]\,\mid\,i\leq h\}$
with  $S[h]\in\mr{Add}Y$ by $\mr{\cite[Lemma~3.3]{AS2}}$. Hence $\mathbf{t}^{'}$ and $\mc{S}$ have the stated shape.

$(2)$  $\mr{Gen}T=T^{\perp_{1}}\subseteq \bigcap\limits_{S\in\mc{U}}S[\infty]^{\perp_{1}}\cap Y^{\perp_{1}}$ since $S[\infty]$, $S\in\mc{U}$, and $Y$ are summands of $T$.

 Now take $X\in \bigcap\limits_{S\in\mc{U}}S[\infty]^{\perp_{1}}\cap Y^{\perp_{1}}$. By Remark~\ref{perp}, it follows that $X$
lies in  $\bigcap\limits_{n\geq 1}S[n]^{\perp_{1}}$ for all $S\in\mc{U}$.
Moreover, if $S[h]\in\mr{Add} Y$, then $X$ lies in $S[h]^{\perp_{1}}$ and also in $S[i]^{\perp_{1}}$ for each $1 \leq i\leq h$,  due to the inclusion $S[i]\hookrightarrow S[h]$.
 This shows that $X\in \mathbf{t}^{'\perp_{1}}=\mc{S}^{\perp_{1}}=\mr{Gen}T$.
\end{proof}

 We want to find a wide subcategory $\mc{M}\subset $ mod$R$ which corresponds to $T$ under the bijection in Theorem \ref{00}.  From {$\mr{\cite[Corollary~4.13]{AS1}}$} and $\mr{\cite[Theorem~2.6]{scho}}$, we know that the tilting module
$ R_{\mc{M}}\oplus R_{\mc{M}}/ R$ given by a wide subcategory $\mc{M}$ has  tilting class $\mr{Gen}(R_{\mc{M}})=\mc{M}^{\perp_{1}}$. Hence $\mc{M}$ must satisfy $\mr{Gen}T=\mc{M}^{\perp_{1}}$, and in particular it must be contained in $\mc{S}$. Furthermore, the class $\mc{M}^{\perp_{1}}$ must be contained in $Y^{\perp_{1}}$ and in
each $S[\infty]^{\perp_{1}}$ with $S\in\mc{U}$.

To construct such $\mc{M}$, we will therefore  pick modules from $\mc{S}$ which filter the Pr\"ufer modules $S[\infty]$ with $S\in\mc{U}$. This will ensure the inclusion $\mc{M}^{\perp_{1}}\subseteq \bigcap\limits_{S\in\mc{U}}S[\infty]^{\perp_{1}}$ by Remark~\ref{perp}.

Moreover, for each ray $\{S[n]\,|\, n\geq 1\}$ containing a module in $\mr{Add}Y$,
we will include in $\mc{M}$ the module $S[h]\in\mr{Add} Y$ of maximal regular length on that ray. This will entail the inclusion  $\mc{M}^{\perp_{1}}\subseteq Y^{\perp_{1}}$, again by Remark~\ref{perp}.  In fact, a closer look shows that it suffices to take care of the rays starting in quasi-simple modules that are not in  $\mc{U}$, because for $S\in\mc{U}$ we already have  $\mc{M}^{\perp_{1}}\subseteq \bigcap\limits_{n\ge 1}S[n]^{\perp_{1}}$.

So, given a wing  $\mc{W}=\mc{W}_{U[m]}$  inside a tube $\mathbf{t}_{\lambda}$, where the vertex $U[m]$ belongs to $\mr{Add}Y$,{ we consider}  the following two sets$$\mc{X}_{\mc{W}}=\{S[h]\,|\,  S[h] \in\text{Add}Y \text{ of maximal regular length on a ray starting in } S\in \mc{W}\},$$
and
$$\widetilde{\mc{X}}_{\mc{W}} = \{S[h] \,|\,  S[h] \in\text{Add}Y \text{ of maximal regular length on a ray starting in } S\in
 \mc{W}\setminus\{U\}\}.$$

We will now follow this strategy and  construct the category $\mc{M}$.

\begin{cons}{\rm
Let $ Q:=\{\lambda\in\mathbb{X}|~\textbf{t}_{\lambda}\cap\mr{Add}Y\neq\emptyset\}$, and fix  $\lambda\in Q$.
Let $\{S_{1},\cdots,S_{r}\}$ be a complete irredundant set of the quasi-simple modules in $\textbf{t}_{\lambda}$. By $\mr{\cite[Proposition~3.7]{AS2}}$,  the summands of $Y$ belonging to $\textbf{t}_{\lambda}$ are arranged in disjoint wings $\mc{W}_{1},\ldots, \mc{W}_{l}$ in $\textbf{t}_{\lambda}$ whose vertices $S_{n_{1}}[m_{1}], \ldots, S_{n_{l}}[m_{l}]$ belong to  $\mr{Add}Y$. In other words,
$$\textbf{t}_{\lambda}\cap\mr{Add}Y\subseteq \bigcup\limits_{j=1}^{l}\mc{W}_{j}.$$
Observe that $S_{n_1},\ldots,S_{n_l}\in\mc{U}$ whenever $\lambda\in Q\cap P$.
We now define the set
$$\mc{X}(\lambda)=\bigcup\limits_{j=1}^{l}\mc{X}_{j}$$ where for each $1\le j\le l$ the set $\mc{X}_{j}$ is given as follows:
$$\left\{
\begin{array}{ll}
 \mc{X}_{j}=\mc{X}_{\mc{W}_j}&\text{if } \lambda\in Q\setminus P,  \\
\mc{X}_{j}=\widetilde{\mc{X}}_{\mc{W}_j} & \text{if } \lambda\in Q\cap P. \end{array}
\right.
$$
When $\lambda\in Q\cap P$, we also need to filter the Pr\"ufer modules  {$\{S[\infty]|~S\in\mc{U}\}$}, so we take
$$\mc{R}(\lambda)=\bigcup\limits_{j=1}^{l}\mc{R}_{j}$$
where
$$\begin{array}{l}
\mc{R}_{j}=
\{S_{n_{j}}[k]\,|\, ~m_{j}< k \leq n_{j+1}-n_{j}\}\cup \{S_{i}| ~n_{j}+m_{j}+1 \leq i < n_{j+1}\}~~ \mr{for} ~ 1\leq j < l,\\
\mc{R}_{l}=
\{S_{n_{l}}[k]\,|\, ~m_{l}< k \leq r+n_{1}-n_{l}\}\cup\{S_{i}| ~n_{l}+m_{l}+1\leq i \leq r \} \cup\{S_{i}| ~1\leq i < n_{1} \}.
\end{array}$$
Now we define
 $$\mc{Q}(\lambda)=\left\{
\begin{array}{ll}
\mc{X}(\lambda)& \mr{if} ~ \lambda\in Q\setminus P\\
\mc{X}(\lambda)\cup \mc{R}(\lambda)& \mr{if} ~ \lambda\in Q\cap P
\end{array}
\right.
$$
and let $\mc{M}={^{\perp}(\mc{Q}^{\perp})}\cap \mr{mod}R$ be the wide closure of the set
$$\mc{Q}=\bigcup\limits_{\lambda\in Q}\mc{Q}({\lambda})\,\cup \bigcup\limits_{\mu\in P\setminus Q}\textbf{t}_{\mu}.$$
}\end{cons}

\begin{exmp}{\rm In Figure 1 we  illustrate the definition of   $\mc{R}_{\lambda}$
by considering the case of $l=2$ wings $\mc{W}_1$ and $\mc{W}_2$ rooted in the vertex $S_1[3]$ and $S_7[4]$, respectively, inside a tube $\tube_\lambda$ of rank $r=12$.
Then $\mc{R}_{1}=\{S_{{1}}[4], S_{{1}}[5], S_{{1}}[6], S_5, S_6\}$, and
$\mc{R}_{2}=\{ S_{7}[5], S_{7}[6], S_{12}\}$.

\begin{figure}[h]
$$
\def\c{\circ} \def\b{\bullet} \def\d{\cdot} \xymatrix@-1.15pc@!R=5pt@!C=5pt{ & & & &
   & & & & & & & & & & & & & &  & &
  &  & \\
  \vdots &\d\ar[rd]& &\d\ar[rd]& &\star\ar[rd]& &\d\ar[rd]& &\d\ar[rd]&
  &\d\ar[rd]& &\d\ar[rd]& &\d\ar[rd]& &\ast\ar[rd]&
  &\d\ar[rd]& &\d\ar[rd]& \vdots \\
  \d\ar[ru]\ar[rd] & &\d\ar[rd] \ar[ru]& &\star\ar[rd]\ar[ru] & &\d\ar[ru]\ar[rd] &
  &\d\ar[ru]\ar[rd]& &\d\ar[ru]\ar[rd] & &\d\ar[ru]\ar[rd]&
  &\d\ar[rd]\ar[ru]& &\ast\ar[rd]\ar[ru]&
  &\d\ar[ru]\ar[rd]& &\d\ar[ru]\ar[rd]& &\d\\
  & \d\ar[rd]\ar[ru]& &{\star}\ar[ru]\ar[rd]& &\d\ar[ru]\ar[rd]&
  &\d\ar[ru]\ar[rd]& & \d\ar[ru]\ar[rd]& &\d\ar[ru]\ar[rd]&
  &\d\ar[rd]\ar[ru]& &\d\ar[rd]\ar[ru]& &\d\ar[ru]\ar[rd]&
  &\d\ar[ru]\ar[rd]& &\d\ar[ru]\ar[rd]& \\
  \d\ar[ru]\ar[rd]& & \b\ar[ru]\ar[rd]& &\d\ar[ru]\ar[rd]&
  &\d\ar[ru]\ar[rd]& &\d\ar[ru]\ar[rd]& &\d\ar[ru]\ar[rd] &
  &\d\ar[rd]\ar[ru]& &\d\ar[ru]\ar[rd]& &\d\ar[ru]\ar[rd]&
  &\d\ar[ru]\ar[rd]& &\d\ar[ru]\ar[rd]& & \d\\
  &\b\ar[ru]\ar[rd]& &\d\ar[ru]\ar[rd]& &\d\ar[ru]\ar[rd]&
  &\d\ar[ru]\ar[rd]& &\d\ar[ru]\ar[rd]& &\d\ar[rd]\ar[ru]& &\d\ar[ru]\ar[rd]&
  &\d\ar[ru]\ar[rd]& &\d\ar[ru]\ar[rd]&
  &\d\ar[ru]\ar[rd]& &\d\ar[rd]\ar[ru]& \\
  \d\ar[ru] \ar @{-}[uuuuu]& &\b\ar[ru] & &\d\ar[ru] &
  &\d\ar[ru] & &\star\ar[ru] & &\star\ar[ru] & & \d\ar[ru]
  & & \d\ar[ru] & & \d\ar[ru] & & \d\ar[ru] & & \d\ar[ru] & &
  \ast \ar @{-}[uuuuu]\\
  S_{{1}}& & & & &
  & & & & & &  &S_{{7}}
  & &  & &  & &  & &  & &S_{12}
  ~
}
$$
\caption{$\bullet=$ direct summands of $Y$ {in the wing $\mc{W}_1$}; $\star$ {and $\ast=$ modules in $\mc{R}_{1}$ and  $\mc{R}_{2}$}}
\end{figure}

Now suppose that $S_1[2], S_1[3], S_2$ are the indecomposable summands of $Y$ lying in the wing $\mc{W}_1$. Then $\mc{X}_1=\{S_2\}$ if $\lambda\in Q\cap P$, while $\mc{X}_1=\{S_1[3], S_2\}$ if $\lambda\in Q\setminus P$.

\medskip

}
\end{exmp}

~

~

~

~

~

\begin{exmp}
{\rm
Let $\mathbf{t}_{\lambda}$ be a tube of rank $r>1$ with quasi-simple modules $S=S_1,\ldots, S_r$.

\smallskip

(1) (cf.~\cite[Example~4.4]{AS2}) Take $T=T_{(Y,P)}$ where $Y=S\oplus\cdots\oplus S[r-1]$ and $P=\{\lambda\}$.
\\
 Then  $\mc{V}=\{S_1,\ldots, S_r\}$
and $\mc{U}=\{S\}$. Moreover, $P=Q$ and $\mc{Q}=\mc{R}(\lambda)=\{S[r]\}$. \\
So, the tilting module
 $$T= S\oplus\cdots\oplus S[r-1]\oplus R_{\mc{V}}\oplus S[\infty]$$ is equivalent to $T_{\mc{M}}=R_{\mc{M}}\oplus R_{\mc{M}}/ R$ for $\mc{M}={^{\perp}(S[r]^{\perp})}\cap\mr{mod}R$.

\medskip

(2)
Let now  $r=3$. Take $T=T_{(Y,P)}$ where $Y=S_1$ and $P=\{\lambda \}$.
\\
 Then  $\mc{V}=\{S_1,S_2, S_3\}$
and $\mc{U}=\{S_1, S_3\}$. Moreover, $P=Q$ and
$\mc{Q}=\mc{R}(\lambda)=\{S_1[2], S_1[3],S_{3}\}$.\\
 So, the tilting module
$$T=S_1\oplus R_{\mc{V}} \oplus S_1[\infty]\oplus S_3[\infty]$$ is equivalent to  $T_{\mc{M}}=R_{\mc{M}}\oplus R_{\mc{M}}/ R$ for
$\mc{M}={^{\perp}(\{S_1[2], S_1[3],S_{3}\}^{\perp})}\cap\mr{mod}R$.

\medskip

{(3)}  Let $r=4${, and let} $\mathbf{t}_{\mu}$  be a tube of rank $r_{\mu}=3$ with  quasi-simples $\{U_{1},U_{2},U_{3}\}$.
Take $T_{(Y,P)}$ where {$Y=S_{1}[3]\oplus S_{2}[2]\oplus S_{2}\oplus U_{1}[2]\oplus U_{2}$ and $P=\{\lambda\}$.} \\
 Then $$ Q=\{\lambda,\mu\},~\mc{X}(\lambda)=\{S_{2}[2]\},
 \mc{X}(\mu)=\{U_1[2],U_2\}~\mr{and}~\mc{R}(\lambda)=\{S_{1}[4]\}.$$
 Moreover, $\mc{Q}=\mc{X}(\lambda)\cup \mc{X}(\mu)\cup \mc{R}(\lambda)$. So, the tilting module $T_{(Y,P)}$ is equivalent to $T_{\mc{M}}=R_{\mc{M}}\oplus R_{\mc{M}}/ R$ for $\mc{M}={^{\perp}(\mc{Q}^{\perp})}\cap\mr{mod}R$.
}\end{exmp}

We collect some easy observations.

\begin{rem}\label{10} {\rm (1)  Since $\mc{Q}\subseteq \bigcup\limits_{\lambda\in P\cup Q}t_{\lambda}$,  its wide closure $\mc{M}$ is contained in $\mr{add}\,(\bigcup\limits_{\lambda\in P\cup Q}t_{\lambda})$.
\\
 (2) By construction, we have  $\bigcup\limits_{j=1}^{l}\mc{W}_{j}\subseteq \mc{R}(\lambda)^{\perp}\,\cap \mathbf{t}_{\lambda}$. Moreover, $\mr{Hom}_{R}(X,Y)\neq0$ for any $X\in\mc{R}(\lambda)$ and $Y\in\mathbf{t}_{\lambda} \setminus\bigcup\limits_{j=1}^{l}\mc{W}_{j}$. It follows that  $\mc{R}(\lambda)^{\perp}\cap \mathbf{t}_{\lambda}=\bigcup\limits_{j=1}^{l}\mc{W}_{j}.$
 }
\end{rem}


\bigskip

The indecomposable modules in $\mc{M}$ are therefore of the form  $M=S[i]\in \tube_{\lambda}$ with $\lambda\in P\cup Q$. If $\lambda\in P\setminus Q$, all modules in $\tube_{\lambda}$ do occur. Let us turn to the remaining  cases.

\begin{prop}\label{4}
Any indecomposable module
in $\mc{M}$ either lies on a ray starting in $\mc{U}$, or there is some $S[h]\in \bigcup\limits_{\lambda\in Q}\mc{X}(\lambda)$ such that $M\cong S[i]$ with $1\le i\le h$.

In particular, $\mc{M}$ is contained in $\mc{S}$.
\end{prop}
\begin{proof}
Let $M\in\mc{M}$ be indecomposable. Then w.l.o.g.~$M=S[i]\in\tube_{\lambda}$ for some $\lambda\in P\cup Q$. If $M$ does not lie on a ray starting in $\mc{U}$,  we have one of the following cases:
(1)  $\lambda\in Q\cap P$ and $S\notin\mc{U}$, or (2)  $\lambda\in Q\backslash P$. We will show that in both cases there is $h\geq i$ such that $S[h]\in\mc{X}(\lambda)$.

\smallskip

(1) Assume that $\lambda\in Q\cap P$ and $S\notin \mc{U}$. By assumption $S[i]\in {^{\perp}(\mc{Q}^{\perp})}\subseteq {^{\perp}(\mc{Q}^{\perp}\cap \mathbf{t}_{\lambda})}$ and $\mc{Q}^{\perp}= \bigcap\limits_{\mu\in Q}\mc{Q}({\mu})^{\perp} \cap (\bigcup\limits_{\mu\in P \setminus Q}\mathbf{t}_{\mu})^{\perp}$.
Now, we have $\mathbf{t}_{\lambda}\subseteq \mc{Q}_{\mu}^{\perp}$, when $\mu\neq \lambda$ and $\mathbf{t}_{\lambda}\subseteq (\bigcup\limits_{\mu\in P \setminus Q}\mathbf{t}_{\mu})^{\perp}$ since $\lambda\in P \cap Q$. Thus $\mc{Q}^{\perp} \cap \mathbf{t}_{\lambda}=\mc{Q}({\lambda})^{\perp} \cap \mathbf{t}_{\lambda}={\mc{R}(\lambda)}^{\perp}\cap \mathbf{t}_{\lambda}\cap {\mc{X}(\lambda)}^{\perp}$.
From Remark \ref{10}(2) we infer that {$S[i]\in {^{\perp}(\bigcup\limits_{j=1}^{l}\mc{W}_{j} \cap \mc{X}({\lambda})^{\perp})}=\bigcap\limits_{j=1}^{l}{^{\perp}(\mc{W}_{j} \cap \mc{X}({\lambda})^{\perp})}=\bigcap\limits_{j=1}^{l}{^{\perp}(\mc{W}_{j} \cap \mc{X}_{j}^{\perp})}$}, keeping in mind that   $\mc{W}_{j} \subseteq \mc{W}_{k}^{\perp} \subseteq\mc{X}_{k}^{\perp}$ whenever $j\not= k$.

By assumption, we further know that $S\notin\mc{U}$ is a composition factor of $\tau^-Y$. So there is $ 1 \leq j \leq l$ such that $S\in\{S_{n_{j}+1},\cdots,S_{n_{j}+m_{j}}\}$. {Notice that not all   modules in Add$Y\cap \mc{W}_{j}$ can lie on the ray starting in $S_{n_j}$, because otherwise $\mc{X}_{j}=\widetilde{\mc{X}}_{\mc{W}_j}=\emptyset$ and the assumption $S[i]\in {^{\perp}(\mc{W}_{j}\cap\mc{X}_{j}{}^{\perp})}={^{\perp}\mc{W}_{j}}$ would yield a contradiction. In particular, this entails $m_j>1$.}
 So we can apply Proposition \ref{2} (1) below and obtain $i \leq h < m_{j}$ such that $S[h]\in\mc{X}({\lambda})\subseteq \mr{Add}Y$.

(2) Assume that $\lambda\in Q\setminus P$. By construction $\mc{X}({\lambda})= \bigcup\limits_{j=1}^{l}\mc{X}_{j} \subseteq \bigcup\limits_{j=1}^{l}\mc{W}_{j}$, so its wide closure  $^{\perp}(\mc{X}({\lambda})^{\perp})\cap \mr{mod}R$ is contained in  $\mr{add}\bigcup\limits_{j=1}^{l}\mc{W}_{j}$. As different tubes are Hom- and Ext-orthogonal,
$\mc{M}\cap \mathbf{t}_{\lambda}= {^{\perp}(\mc{X}({\lambda})^{\perp})\cap \mathbf{t}_{\lambda}} \subseteq \mr{add} \bigcup\limits_{j=1}^{l}\mc{W}_{j}$.
Since $S[i]\in\mc{M}\cap  \mathbf{t}_{\lambda}$, there is $1 \leq j\leq l$ such that $S[i]\in \mc{W}_{j}$, thus $S\in\{S_{n_{j}},\cdots,S_{n_{j}+m_{j}-1}\}$.
Denote $\mc{W}^{}_{0}=\mc{W}_{\tau S_{n_{j}}[m_{j}+1]}$.
 Note that $\mc{W}^{}_{0} \subseteq \mc{W}_{k}^{\perp} \subseteq\mc{X}_{k}^{\perp}$ whenever $j \neq k$. Hence
 $S[i]\in  {^{\perp}(\mc{X}({\lambda})^{\perp})} \subseteq  {^{\perp}(\mc{W}^{}_{0} \cap \mc{X}_{j}^{\perp})}$. Now we can apply Proposition \ref{2} (2) below and obtain $i \leq h \leq m_{j}$ such that $S[h]\in\mc{X}({\lambda})\subseteq \mr{Add}Y$.
\end{proof}





\begin{lem}\label{33}
Let   $\mc{W}=\mc{W}_{U_{1}[m]}$ be a wing rooted in the vertex $U_{1}[m]\in\mr{Add}Y$.
Set  $U_{2}=\tau^{-1}U_{1},\cdots, U_{m+1}=\tau^{-1}U_{m}$. Denote further $\mc{A}={^{\perp}(\mc{W}\cap \widetilde{\mc{X}}_{\mc{W}}{}^{\perp})}$. Then the following  holds true.


(1) $U_{1}[m]\in \mc{W}\cap \widetilde{\mc{X}}_{\mc{W}}{}^{\perp}$.

(2) Assume that $m>1$, and that there are a subset ${\mc{X}}\subseteq \mathbf{t}_{\lambda}$ and integers  $1\leq j< m$ and $0< n \le m-j$
such that $\mc{W}^{'}:=\mc{W}_{U_{j}[n]}\subset{\mc{X}}^\perp$ and ${\widetilde{\mc{X}}}_{\mc{W}}={\widetilde{\mc{X}}}_{\mc{W}^{'}}\cup {\mc{X}}$.
 Then $\mc{A}
 \subseteq{^{\perp}(\mc{W}^{'}\cap{\widetilde{\mc{X}}}_{\mc{W}^{'}}{}^{\perp})}$.


\end{lem}
\begin{proof}
(1) is obvious. For (2) note that $\mc{W}^{'}\cap{\widetilde{\mc{X}}}_{\mc{W}^{'}}{}^{\perp}\subseteq \mc{W}\cap{\widetilde{\mc{X}}}_{\mc{W}^{'}}{}^{\perp}\cap{\mc{X}}^{\perp}=\mc{W}\cap{\widetilde{\mc{X}}}_{\mc{W}}{}^{\perp}$. Then $\mc{A}\subseteq {^{\perp}(\mc{W}^{'}\cap{\widetilde{\mc{X}}}_{\mc{W}^{'}}{}^{\perp})}$.
\end{proof}
\begin{prop}\label{2}
Let assumptions and notation be as in Lemma~\ref{33}.

(1)  Assume that $m>1$ and that there are $i\in \mathbb{N}$ and  $S\in\{U_{2},\cdots, U_{m},U_{m+1}\}$  such that $S[i]\in\mc{A}={^{\perp}(\mc{W}\cap \widetilde{\mc{X}}_{\mc{W}}{}^{\perp})}$.
Then $S\neq U_{m+1}$ and there is $i \leq h < m$ such that $S[h]\in\mr{Add}Y$.

(2) Denote  $\mc{W}_{0}=\mc{W}_{\tau U_{1}[m+1]}$, and
assume  that there are $i\in \mathbb{N}$ and $S\in\{U_{1},U_{2},\cdots, U_{m}\}$ such that $S[i]\in \mc{B}={}^{\perp}(\mc{W}_{0}\cap \mc{X}_{\mc{W}}{}^{\perp})$.
Then  there is $i \leq h \leq m$ such that $S[h]\in\mr{Add}Y$.
\end{prop}
\begin{proof}
(1) Since   $\mc{A}\subseteq {^{\perp}U_{1}[m]}$ by Lemma \ref{33}, we have  $S[i]\in\mc{W}$ and $S\in\{U_{2},\cdots,U_{m}\}$.
{Observe that not all modules in Add$Y\cap \mc{W}$ can lie on the ray starting in $U_1$, as we have already seen in the proof of Proposition~\ref{4}.}

We  proceed  by induction on $m$.
{When $m=2$, we know by the observation above
 that $
\mc{W}\cap \mr{Add}Y=\{U_{2}, U_{1}[2]\}$.} Then $\mc{A}={^{\perp}(\mc{W}\cap U_{2}^{\perp} ) } $ does not intersect  $\{U_{2}[t]| ~t \geq 2\}$. Thus $S=U_{2}\in\mr{Add}Y$.

Let $m>2$. We distinguish two cases.

Case (i). Suppose that none of the modules $U_{2}[m-1],U_{3}[m-2],\cdots~,U_{m}$ on the right border of the wing $\mc{W}$ belongs to $\mr{Add}Y$. By the definition of a branch module, we know that  $\mr{Add}Y $ contains precisely $m$ modules in $\mc{W}$: these are $U_{1}[m]$ and  $m-1$ modules in  $\mc{W}^{'}=\mc{W}_{U_{1}[m-1]}$.
 We obtain  $\widetilde{\mc{X}}_{\mc{W}}={\widetilde{\mc{X}}}_{\mc{W}^{'}}$.
By Lemma \ref{33} (here  $\mc{X}=\emptyset$), we have $\mc{A}\subseteq {^{\perp}(\mc{W}^{'}\cap \widetilde{\mc{X}}_{\mc{W^{'}}}{}^{\perp} )}$. By induction assumption,  the claim holds true for all
 $S\in\{U_{2},\cdots, U_{m}\}$.

Case (ii). Suppose that one of the modules $U_{2}[m-1],U_{3}[m-2],\cdots~,U_{m}$ belongs to $\mr{Add}Y$. Choose $U_{k}[h]\in\mr{Add}Y$ of  maximal regular length.  By the definition of a branch module,   the module $U_{1}[t_{1}]$  with $t_{1}=m-h-1=k-2$ must lie in $\mr{Add}Y$.
Furthermore, $\mr{Add}Y $ contains precisely $m$ modules in $\mc{W}$: these are $t_{1}$ modules from $\mc{W}_{1}:=\mc{W}_{U_{1}[t_{1}]}$, together with
$h$ modules from  $\mc{W}_{2}:=\mc{W}_{U_{k}[h]}$ and $U_{1}[m]$.

By induction assumption and Lemma \ref{33} (here $\mc{W}^{'}=\mc{W}_{1}$ {and} $\mc{X}=\mc{X}_{\mc{W}_{2}}$), the claim holds true for $S\in\{U_{2},\cdots,U_{t_{1}},U_{k-1}\}$.
Moreover, the claim is clear for $S=U_{k}$.

We  again have two cases:

Case (i). Suppose that  none of $U_{k+1}[h-1], U_{k+2}[h-2],\cdots,U_{m}$ belongs to $\mr{Add}Y$. Then $U_{k}[h-1]\in \mr{Add}Y$. The claim holds true for all $S\in\{U_{k+1},\ldots,U_m\}$ again by Lemma \ref{33} (here $\mc{W}^{'}=\mc{W}_{U_{k}[h-1]}$ {and}  $\mc{X}= \widetilde{\mc{X}}_{{\mc{W}_{1}}}\cup\{U_{k}[h]\}$) and induction assumption.

Case (ii). Suppose that one  of $U_{k+1}[h-1], U_{k+2}[h-2],\cdots,U_{m}$ belongs to  $\mr{Add}Y$. Choose $U_{k_{2}}[h_{2}]\in\mr{Add}Y$ of maximal regular length. Then   $U_{k}[t_{2}]$ with  $t_{2}=h-h_{2}-1$ must lie in  $\mr{Add}Y$, and the  modules in $\mc{W}_2\cap\mr{Add}Y$ different from the vertex are  distributed in the wings $\mc{W}_3:=\mc{W}_{U_{k}[t_{2}]}$  and $\mc{W}_4:=\mc{W}_{U_{k_{2}}[h_{2}]}$. By Lemma \ref{33} (here $\mc{W}^{'}=\mc{W}_{3}$ {and} $\mc{X}=\widetilde{\mc{X}}_{\mc{W}_{1}}\cup\mc{X}_{{\mc{W}_{4}}}\cup\{U_{k}[h]\}$) and induction assumption, we obtain the claim for $S\in\{U_{k+1},\cdots,U_{k_{2}-1}\}$. Moreover, the claim is clear if $S=U_{k_{2}}$. We proceed in this way, obtaining $$2\leq k=k_{1}< k_{2}< k_{3}<\cdots \leq m,$$ and we keep distinguishing the cases $(i),(ii)$.
After a finite number, say $s$, steps, either  we reach $k:=k_{s}$
 where case (i)  occurs,
  and then we are done, or we reach $k:=k_{s}\leq m$ such that for $h:=h_{s}$ with $k+h=m+1$ we have $U_{k}[h]\in\mr{Add}Y$ and all modules on the coray of $U_{m}$ of regular length $< h$ are in Add$T$. Notice that in the step before, it  only remained to prove the claim for $S\in\{U_{k},\cdots,U_{m}\}$. But here the claim is trivial, as the intersection of the ray of $S$ with the coray of $U_{m}$ is in Add$T$.





\medskip


(2)
Observe that $\tau U_1[m+1]\in \mc{W}_{0}\cap \mc{X}_{\mc{W}}{}^{\perp}$, hence $\mc{B}\subseteq {^{\perp}\tau U_1[m+1]}$. So, $S[i]\in\mc{W}$.
Notice  that the claim is clear for $S=U_{1}$. So, we can assume $m>1$ and $S\in\{U_{2},\cdots,U_{m}\}$.
 We distinguish two cases.

(i) Suppose that none of the modules  $U_{m},U_{m-1}[2],\cdots,U_{2}[m-1]$ belongs to $\mr{Add}Y$.   Then $U_{1}[m-1]\in\mr{Add}Y$.
Let $\mc{W}^{'}=\mc{W}_{U_{1}[m-1]}$. Then $\mc{X}_{\mc{W}}=\widetilde{\mc{X}}_{\mc{W}^{'}}\cup \{U_{1}[m]\}$.
It follows that $\mc{B}\subseteq{^{\perp}(\mc{W}^{'}\cap{\widetilde{{\mc{X}}}}_{\mc{W}^{'}}{}^{\perp})}$. By statement (1), the claim holds for $S\in\{U_{2},U_{3},\cdots,U_{m}\}$.

(ii)
 Suppose that one of the modules  $U_{m},U_{m-1}[2],\cdots,U_{2}[m-1]$ belongs to $\mr{Add}Y$.  Choose $U_{k}[h]\in\mr{Add}Y$ of  maximal regular length. Then  $U_{1}[t_{1}]$  with $t_{1}=m-h-1=k-2$ must lie in $\mr{Add}Y$.
  So, $\mr{Add}Y $ contains precisely $m$ modules in $\mc{W}$: these are the $t_{1}$ modules in $\mc{W}_{1}=\mc{W}_{U_{1}[t_{1}]}$,  together with  $h$ modules from $\mc{W}_{2}=\mc{W}_{U_{k}[h]}$ and $U_{1}[m]$.

 Notice that $\mc{W}_{1}\cap{\widetilde{\mc{X}}}_{\mc{W}_{1}}{}^{\perp}\,\subseteq \mc{W}_{0}\cap{\widetilde{\mc{X}}}_{\mc{W}_{1}}{}^{\perp}\cap \mc{X}_{\mc{W}_{2}}{}^{\perp}\cap U_{1}[m]^{\perp} =\mc{W}_{0}\cap  \mc{X}_{\mc{W}}{}^{\perp} $. Therefore   $\mc{B}\subseteq{^{\perp}(\mc{W}_{1}\cap{\widetilde{\mc{X}}}_{\mc{W}_{1}}{}^{\perp})}$ and the claim follows from statement (1) for $S\in\{U_{2},\cdots,U_{k-1}\}$. Moreover, the claim is clear if $S=U_{k}$. So, it remains to verify the claim when $S\in\{U_{k+1},\cdots,U_{m}\}$.

We  {again have} two cases:

(i) Suppose none of $U_{k+1}[h-1], U_{k+2}[h-2],\cdots,U_{m}$ belongs to $\mr{Add}Y$.
Then $U_{k}[h-1]\in \mr{Add}Y$. It follows that $\mc{B}\subseteq{^{\perp}(\mc{W}^{''}\cap\widetilde{\mc{X}}_{\mc{W}^{''}}{}^{\perp})}$, here $\mc{W}^{''}=\mc{W}_{U_{k}[h-1]}$.
Then the claim follows again from statement (1).

(ii) Suppose one  of $U_{k+1}[h-1], U_{k+2}[h-2],\cdots,U_{m}$ belongs to $\mr{Add}Y$. {Choose $U_{k_{2}}[h_{2}]\in\mr{Add}Y$ of maximal regular length. Then   $U_{k}[t_{2}]$ with  $t_{2}=h-h_{2}-1$ must lie in $\mr{Add}Y$.} It follows that $\mc{B}\subseteq{^{\perp}(\widetilde{\mc{W}}\cap\widetilde{\mc{X}}_{\widetilde{\mc{W}}}{}^{\perp})}$ where $\widetilde{\mc{W}}=\mc{W}_{U_{k}[t_{2}]}$. By statement (1), we obtain the claim for $S\in\{U_{k+1},\cdots,U_{k_{2}-1}\}$. The claim is clear if $S=U_{k_{2}}$.

 Continuing in this fashion  and {arguing} as above,  one obtains the claim for the remaining quasi-simple mosules   $S\in\{U_{k_{2}+1},\cdots,U_{m}\}$.
\end{proof}



Next, we have to show that the Pr\"ufer modules corresponding to quasi-simples in $\mc{U}$ are filtered by modules from $\mc{M}$.


\begin{prop}\label{6}
 If $S\in \mc{U}$ belongs to a tube $\mathbf{t}_{\lambda}$ of rank $r$, then $S[r]\in\mc{M}$.
\end{prop}
\begin{proof}
The claim is clear for $\lambda\in P\setminus Q$ by construction of $\mc{M}$.
So we assume that $\lambda\in P \cap Q$.
Recall that the summands of $Y$ are arranged in the  wings $\mc{W}_1,\dots,\mc{W}_l$ rooted in $S_{n_1}[m_{1}],\ldots,S_{n_l}[m_{l}]$. Since $S$ is not a composition factor of $\tau^-Y$, it either  {lies in $\{S_{n_j},S_{m_{j}+n_{j}+1},\cdots,S_{n_{j+1}-1}\}$ for some $1 \leq j< l$, or in $\{S_{n_l}, S_{n_{l}+m_{l}+1},\ldots,S_{r},S_{1},\ldots,S_{n_{1}-1}\}$.
We can assume w.l.o.g.~that $S$ lies in $\{S_{n_1},S_{m_{1}+n_{1}+1},\cdots,S_{n_{2}-1}\}$}. We will proceed case by case and express $S[r]$ as an iterated extension of modules from $\mc{R}(\lambda)$.

\medskip

Case (i): $S=S_{n_{1}}$. We consider the following short exact sequences
$$0\rightarrow S_{n_{1}}[n_{2}-n_{1}]\rightarrow S_{n_{1}}[r]\rightarrow S_{n_{2}}[r+n_{1}-n_{2}]\rightarrow 0$$
$$0\rightarrow S_{n_{2}}[n_{3}-n_{2}]\rightarrow S_{n_{2}}[r+n_{1}-n_{2}]\rightarrow S_{n_{3}}[r+n_{1}-n_{3}]\rightarrow 0$$
$$\vdots~$$
$$0\rightarrow S_{n_{l-1}}[n_{l}-n_{l-1}]\rightarrow S_{n_{l-1}}[r+n_{1}-n_{l-1}]\rightarrow S_{n_{l}}[r+n_{1}-n_{l}]\rightarrow 0$$
Since $S_{n_{l}}[r+n_{1}-n_{l}]$ and all the first terms of these exact sequences are in $\mc{R}({\lambda})$, we obtain that $S_{n_{1}}[r]\in{\mc{M}}$.


\medskip

Case (ii):  $S=S_{n_{1}+m_{1}+1}$.
Notice that  $\mc{M}$, being extension closed, contains
 the wing  $\mc{W}^{'}$ rooted in ${S_{n_{1}+m_{1}+1}[n_{2}-n_{1}-m_{1}-1]}$.
Consider the exact sequence
$$0\rightarrow S_{n_{1}+m_{1}+1}[n_{2}-n_{1}-m_{1}-1] \rightarrow S_{n_{1}+m_{1}+1}[n_{3}-n_{1}-m_{1}-1]\rightarrow S_{n_{2}}[n_{3}-n_{2}]\rightarrow 0.$$
It follows that the middle term $S_{n_{1}+m_{1}+1}[m_{3}-m_{1}-n_{1}-1]$ is in $\mc{M}$. There is an exact sequence
$$0 \rightarrow S_{n_{1}+m_{1}+1}[n_{3}-n_{1}-m_{1}-1]\rightarrow S_{n_{1}+m_{1}+1}[r]\rightarrow S_{n_{3}}[r+n_{1}+m_{1}+1-n_{3}]\rightarrow 0.$$
Now we need to prove that $ S_{n_{3}}[r+n_{1}+m_{1}+1-n_{3}]\in\mc{M}$.
We have the following exact sequences
$$0\rightarrow S_{n_{3}}[n_{4}-n_{3}]\rightarrow S_{n_{3}}[r+n_{1}+m_{1}+1-n_{3}]\rightarrow S_{n_{4}}[r+n_{1}+m_{1}+1-n_{4}]\rightarrow0 $$
$$0\rightarrow S_{n_{4}}[n_{5}-n_{4}]\rightarrow S_{n_{4}}[r+n_{1}+m_{1}+1-n_{4}]\rightarrow S_{n_{5}}[r+n_{1}+m_{1}+1-n_{5}]\rightarrow0 $$
$$\vdots~$$
$$0\rightarrow S_{n_{l}}[r+n_{1}-n_{l}]\rightarrow S_{n_{l}}[r+n_{1}+m_{1}+1-n_{l}]\rightarrow S_{n_{1}}[m_{1}+1]\rightarrow0 $$
Since $ S_{n_{1}}[m_{1}+1]$  and $S_{n_{l}}[r+n_{1}-n_{l}]$ are in $\mc{R}({\lambda})$, we have that $S_{n_{l}}[r+n_{1}+m_{1}+1-n_{l}]\in\mc{M}$.
It follows that $S_{n_{3}}[r+n_{1}+m_{1}+1-n_{3}]\in\mc{M}$. So, $S_{n_{1}+m_{1}+1}[r]\in \mc{M}$.

Case (iii): $S=S_{n_{2}-1}$.
Consider the exact sequence
$$0\rightarrow S_{n_{2}-1}\rightarrow S_{n_{2}-1}[n_{3}-n_{2}+1]\rightarrow S_{n_{2}}[n_{3}-n_{2}]\rightarrow 0.$$
Since $S_{n_{2}}[n_{3}-n_{2}]$ and $S_{n_{2}-1}$ are in $\mc{M}$, we obtain $S_{n_{2}-1}[n_{3}-n_{2}+1]\in\mc{M}$.
There is an exact sequence
$$0\rightarrow S_{n_{2}-1}[n_{3}-n_{2}+1]\rightarrow  S_{n_{2}-1}[r]\rightarrow S_{n_{3}}[r-1+n_{2}-n_{3}]\rightarrow 0. $$
Now, we need to prove $S_{n_{3}}[r-1+n_{2}-n_{3}]\in{\mc{M}}$. We have a series of exact sequences
$$0\rightarrow S_{n_{3}}[n_{4}-n_{3}]\rightarrow S_{n_{3}}[r-1+n_{2}-n_{3}]\rightarrow S_{n_{4}}[r-1+n_{2}-n_{4}]\rightarrow 0 $$
$$\vdots~$$
$$0\rightarrow S_{n_{l-1}}[n_{l}-n_{l-1}]\rightarrow {S_{n_{l-1}}[r-1+n_{2}-n_{l-1}]}\rightarrow S_{n_{l}}[r-1+n_{2}-n_{l}]\rightarrow 0 $$
$$0\rightarrow S_{n_{l}}[r+n_{1}-n_{l}]\rightarrow S_{n_{l}}[r-1+n_{2}-n_{l}]\rightarrow S_{n_{1}}[n_{2}-n_{1}-1]\rightarrow 0.$$
Since $S_{n_{l}}[r+n_{1}-n_{l}]$ and $S_{n_{1}}[n_{2}-n_{1}-1]$ are in $\mc{R}({\lambda})$, we have $S_{n_{l}}[r-1+n_{2}-n_{l}]\in\mc{M}$.
It follows that $S_{n_{3}}[r-1+n_{2}-n_{3}]\in{\mc{M}}$. So, $S_{n_{2}-1}[r]\in{\mc{M}}$.

The  remaining cases  $S\in\{S_{n_{1}+m_{1}+2},\cdots,S_{n_{2}-2}\}$ are obtained similarly, keeping in mind that  $\{S_{n_{1}}[i]\,| ~m_{1}< i\leq n_{2}-n_{1}\}\subseteq \mc{M}$.\end{proof}


\bigskip

We are now ready to prove that  $T$ arises from the universal localization $R\rightarrow R_{\mc{M}}$.

\begin{thm}\label{8}
$T_{(Y,P)}$ is equivalent to $T_{\mc{M}}= R_{\mc{M}}\oplus R_{\mc{M}}/ R$.
\end{thm}

\begin{proof}
 By Lemma \ref{7}, we  need to prove $\bigcap\limits_{S\in\mc{U}}S[\infty]^{\perp_{1}}\bigcap Y^{\perp_{1}}=\mc{M}^{\perp_{1}}$.

Firstly,  $\mc{M}\subseteq \mc{S}$ by Proposition \ref{4}, hence
$\bigcap\limits_{S\in\mc{U}}S[\infty]^{\perp_{1}}\cap Y^{\perp_{1}}=\mc{S}^{\perp_{1}}\subseteq \mc{M}^{\perp_{1}}$.

Secondly, if $S\in \mc{U}$ lies in a tube of rank $r$, then $\mc{M}$ contains $S[r]$ by Proposition \ref{6}, and  we deduce inductively from the short exact sequence
$$0\rightarrow S[r(n-1)]\rightarrow S[nr]\rightarrow S[r]\rightarrow 0 $$
that $S[nr]$ belongs to $\mc{M}$ for all $n\in \mathbb{N}$.
Since $S[\infty]$ is filtered by the $S[nr],\, n\ge 1$, we conclude from Remark~\ref{perp} that
$\mc{M}^{\perp_{1}}
\subseteq
\bigcap\limits_{S\in\mc{U}}S[\infty]^{\perp_{1}}$.

Thirdly, we recall from $\mr{\cite[Proposition~4.2]{AS2}}$ that there is a decomposition
$Y=\bigoplus\limits_{\lambda\in Q}\mathbf{t}_{\lambda}(Y)$ with $\mathbf{t}_{\lambda}(Y)\in \mr{add}\mathbf{t}_{\lambda}$.
If $\lambda\in Q\setminus P$, then the indecomposable summands of $\mathbf{t}_{\lambda}(Y)$ have the form $S[i]$ such that there is $h\ge i$ with $S[h]\in\mc{X}({\lambda})$. In other words, they are submodules of a module from $\mc{X}({\lambda})$.
Similarly, when $\lambda\in Q\cap P$, the indecomposable summands of $\mathbf{t}_{\lambda}(Y)$ are either submodules of a module from $\mc{X}({\lambda})$, or lie on a ray starting in some quasi-simple $S\in\mc{U}$.



Thus $\mc{M}^{\perp_{1}}\subseteq \bigcap\limits_{\lambda\in Q}\mathbf{t}_{\lambda}(Y)^{\perp_{1}}=Y^{\perp_{1}}$, again by Remark~\ref{perp}.
This concludes the proof.
\end{proof}

Combining Theorem \ref{8} and Proposition \ref{101}, we obtain the main result of this section.


\begin{thm}\label{main}
For any pair $(Y,P)$, the tilting module $T_{(Y,P)}$ is minimal if and only if the set $P\subset\mathbb{X}$ is non-empty.
\end{thm}

{We correct a result from \cite{AS2} which contained an error.
\begin{cro}{\rm (cf.~\cite[Corollary 5.10]{AS2})} Let $T$ be  a large tilting $R$-module. The following statements are equivalent.
\begin{enumerate}
\item[(i)] There exists an injective pseudoflat ring epimorphism $\lambda:R\to S$ such that $S\oplus S/R$ is a tilting module equivalent to $T$.
\item[{(ii)}] $T$ is equivalent to a tilting module $T_{\mc{U}}= R_{\mc{U}}\oplus R_{\mc{U}}/R$ where $\mc{U}\subset \tube$ is a set of finite dimensional indecomposable regular (not necessarily quasi-simple) modules.
\end{enumerate}
\end{cro}
}

\subsection{Minimal cotilting modules}
The large cotilting modules have been classified in \cite{Buan}. They are also parametrized by the pairs $(Y,P)$ where $Y$ is a  branch module and $P$ is a subset of $\mathbb{X}$. More precisely, the following modules {form} a complete irredundant list of all large cotilting {left $R$-}modules, up to equivalence {(cf.~\cite[\S 8.5]{A})}:
\begin{verse}
 $C_{(Y,P)}=Y\oplus G\oplus \bigoplus \{\mr{all}~S[\infty]~ \mr{in}~ {^{\perp_{1}}Y} \mr{from~tubes}~\mathbf{t}_{x},x\notin P\}\oplus~~~~~~~~~~~~~~~~~~~~~~~~~~~~$
\ \hskip 101pt ${~~~~~~~\prod \{\mr{all}~S[-\infty]~ \mr{in}~ {Y^{\perp_{1}}} \mr{from~tubes}~\mathbf{t}_{\lambda},\lambda\in P\}}$
\end{verse}

In fact, the duality $D=\mr{Hom}_k(-,k)$ induces a bijection between (equivalence classes of) {tilting right $R$-modules and cotilting left $R$-modules}, which restricts to a bijection between minimal tilting and minimal cotilting modules {(see e.g.~\cite[\S 6.1 and 6.2]{A} and~Theorem~\ref{00})}. As shown in \cite[Appendix]{AS2}, this bijection maps $T_{(Y,P)}$ to $C_{(Y,P)}$.

{Notice that} a cotilting module is large if and only if it has an infinite dimensional indecomposable direct summand {\cite[Lemma 2.6]{AS2}}.

\begin{thm}\label{main2}
A large cotilting module  is a minimal cotilting module if and only if it   has an adic module as a direct summand.
\end{thm}

\begin{proof}
  By  {Theorems~\ref{00} and~\ref{main}},  a large cotilting module
$C$ is minimal if and only if  $C=C_{(Y,P)}$ is equivalent to $D(T_{(Y,P)})$ for   a tilting module $T=T_{(Y,P)}$ given by a pair $(Y,P)$ with $P\neq\emptyset$.
Notice that $C_{(Y,P)}$ can be rewritten as follows
 $$C_{(Y,P)}=Y\oplus G\oplus \bigoplus \{\mr{all}~S[\infty]~ \mr{in}~ {^{\perp_{1}}Y} \mr{from~tubes}~\mathbf{t}_{\lambda},\lambda\notin P\}\oplus
 \prod_{S\in\mc{U}} S[-\infty]$$
and  $P\neq \emptyset$
 if and only if $\mc{U}\neq \emptyset$, which amounts to $C$ having an adic  direct summand.
\end{proof}

{\section{Ascent of minimality}}

The behaviour of (co)silting modules under ring extensions has been studied  in
\cite{BREAZ}. In particular, it was shown that over commutative rings every silting module extends to a silting module.

\begin{thm}\label{b}{\rm\cite[Theorems 2.2 and 2.7]{BREAZ}} Let  $\lambda:R\to S$ be a ring homomorphism, and
let $T$ be a silting module with respect to a projective presentation $\sigma$. Then  $T\otimes_RS$ is a silting $S$-module with respect to $\sigma\otimes_R S$ if and only if   $T\otimes_R S$, viewed as an $R$-module, lies in the silting class $\mr{Gen}T$.

The latter condition is verified whenever $R$ and $S$ are commutative rings.
\end{thm}

 Of course, if $\lambda:R\rightarrow S$ is surjective with kernel $I$, then  every silting right
$R$-module $T$ verifies the condition $T\otimes_{R}S\simeq T/TI\in\mr{Gen}T$ and therefore extends to a silting $S$-module $T\otimes_R S$, cf.~\cite[Corollary 2.4]{BREAZ}.
In general, however, the condition in the theorem above can fail, even when $T$ is minimal.

\begin{exmp}\label{fail}{\rm Let $R$ be the Kronecker algebra, and let $T$ be  a silting $R$-module. Then
 the  $R$-module $T\otimes_{R}S$ lies in $\mr{Gen}T$ for every homological ring epimorphism $\lambda: R\rightarrow S$ if and only if $T$ is not equivalent to the simple projective module $P_{1}$.

\smallskip

The statement follows from the classification results reviewed in Section~\ref{Kronecker}. We proceed in several steps.

\smallskip

{\bf Step 1.} Let $T=P_{1}$, and let $id_{R}\neq\lambda:R\rightarrow S$ be an injective universal localization. Then the associated bireflective subcategory $\mc{X}$ coincides with Add$P_{i} $ for some $i\geq 2$, or with Add$Q_{i+1}$ for some $i\geq 1$, or with $\mc{U}^{\perp}$ for a non-empty set $\mc{U}$ of simple regular modules.

Notice that in all cases there is a non-trivial map from $T$ to a module in $\mc{X}$. This is clear in the first two cases, and in the third we can for instance choose the embedding of $P_{1}$ in the generic module $G\in\mc{X}$. It follows that the $\mc{X}$-reflection $T\rightarrow T\otimes_{R}S$ is non-trivial.

We claim that $T\otimes_{R}S \neq 0$ does not lie in $\mr{Gen}T$. Indeed, no non-trivial module
in $\mc{X}$ can belong to $\mr{Gen}T=\mr{Add}P_{1}$. Again, this is clear in the first two cases,
and in the third we observe that by the Auslander-Reiten formula $\mr{Ext}^{1}_{R}(S,P_{1})\cong D\mr{Hom}_{R}(P_{1},S)\neq 0$ for every simple regular module $S$.

\smallskip

{\bf Step 2.}
Let now $T=Q_{1}$, and let $\lambda:R \rightarrow S$ and $\mc{X}$ be as above. Then there are no nontrivial maps from $T$ to $\mc{X}$, which is again clear in the first two cases, and in the
third it follows from the fact that $\mr{Hom}_{R}(S,Q_{1})\neq 0$ for every simple regular module
$S$. We conclude that the $\mc{X}$-reflection $T\rightarrow T\otimes_{R} S=0$ is trivial, and $T\otimes_{R}S$ lies in $\mr{Gen}T$.

\smallskip

{\bf Step 3.}
By the discussion above, it remains to prove the if-part of the statement, and w.l.o.g.~we can
assume that $\lambda$ is injective, and T is a tilting module. We have one of the following
cases.

{\bf Case (i).} $\lambda$ is the universal localization at $P_{i+1},i\geq 2$, and $\mc{X}=\mr{Add}P_{i}$.

If $\mr{Hom}_{R}(T,P_{i})=0$, then the $\mc{X}$-reflection $T\rightarrow T\otimes _{R}S=0$ is trivial, and
the claim holds true. If Hom$_{R}(T,P_{i})\neq 0$, then T must be preprojective, because
the non-preprojective silting modules all belong to the class $\mr{Gen}\mathbf{L}={^{\perp_0}\mathbf{p}}$. More
precisely, $T$ must be equivalent to $P_{j}\oplus P_{j+1}$ for some $j\leq i$, hence it generates $P_{i}$,
and therefore also $T\otimes_{R}S$ which belongs to $\mc{X}=\mr{Add}P_{i}$.

{\bf Case (ii).} $\lambda$ is the universal localization at $Q_{i},i\geq 1$, and $\mc{X}=\mr{Add}Q_{i+1}$.

As above we can assume w.l.o.g. that Hom$_{R}(T,Q_{i+1})\neq 0$. If T is preinjective,
then it must be equivalent to $Q_{j}\oplus Q_{j+1}$ for some $j \geq i$, hence it generates $Q_{i+1}$,
and therefore also $T\otimes_{R}S$ which belongs to $\mc{X}=\mr{Add}Q_{i+1}$.

Now assume that T is not preinjective. We know from the classification that
the non-preinjective silting modules all belong to the class $\mathbf{q}^{\perp_0}$ of modules without
non-trivial maps from the preinjective component $\mathbf{q}=\{Q_{1},Q_{2},\cdots\}$, which by the
Auslander-Reiten-formula coincides with the class $^{\perp_{1}}\mathbf{q}$.
In particular, Ext$^{1}_{R}(T,Q_{i+1})=0$, and $Q_{i+1}$ belongs to the tilting class Gen$T${.}
Hence $T$ generates $T\otimes_{R}S$.

{\bf Case (iii).} $\lambda$ is the universal localization at a set $\mc{U}\neq\emptyset$ of simple regulars, and $\mc{X}=\mc{U}^{\perp}$.

Again we assume w.l.o.g. that Hom$_{R}(T,\mc{X})\neq0$. Notice that this implies that
T is not preinjective since $\mc{U}^{\perp}\subset \mathbf{q}^{\perp_0}$. Moreover, $\mc{U}^{\perp}\subset {^{\perp_0}\mathbf{p}}$, which by the Auslander-Reiten formula coincides with the class $\mathbf{p}^{\perp_{1}}$. Hence
Ext$^{1}_{R}(T,T\otimes_{R}S)=0$ and $T\otimes_{R}S$ is $T$-generated whenever $T$ is preprojective.
Moreover, $T\otimes_{R}S\in{^{\perp_0}\mathbf{p}}=\mr{Gen}\mathbf{L}$ is also $T$-generated when $T$ is equivalent to the
Lukas tilting module.

It remains to check the case when $T=R_{\mc{V}}\oplus R_{\mc{V}}/R$ for a non-empty set $\mc{V}$
of simple regular modules. Then Gen$T=\mc{V}^{\perp_{1}}$ consists of the modules $X$ with
Ext$^{1}_{R}(V,X)=0$ for all $V\in\mc{V}$. Pick a module $V\in \mc{V}$. If $V$ also belongs to $\mc{U}$, then
Ext$^{1}_{R}(V,T\otimes_{R}S)=0$.  If $V$ does not belong to $\mc{U}$, then it belongs to $\mc{U}^{\perp}=\mc{X}$, as  different tubes are Hom- and Ext-orthogonal. Since $R_{\mc{V}}\rightarrow R_{\mc{V}}\otimes_{R}S$ is an $\mc{X}$-reflection, we infer that Hom$_{R}(R_{\mc{V}}\otimes_{R}S,V)\cong \mr{Hom}_{R}(R_{\mc{V}},V)$, and by the Auslander-Reiten formula Ext$^{1}_{R}(V,R_{\mc{V}}\otimes_{R}S)\cong \mr{Ext}^{1}_{R}(V,R_{\mc{V}})=0$. We conclude that $R_{\mc{V}}\otimes_{R}S$ is $T$-generated, and so are $R_{\mc{V}}/R\otimes_{R}S$ and $T\otimes_{R}S=(R_{\mc{V}}\otimes_{R}S)\oplus(R_{\mc{V}}/R\otimes_{R}S)$.
}\end{exmp}


Now we return to the commutative case. If $\lambda:R\to S$ is a ring epimorphism and $R$ is commutative, then so is $S$ by \cite{Silver}. It follows from Theorem~\ref{b} that every silting module extends to a silting module along $\lambda$.
In fact, when $R$ is commutative and hereditary, and $\lambda$ is pseudoflat,  also minimality is  preserved. Indeed, $S$ is then hereditary as well by \cite[p.324]{BD}, and all silting modules over commutative hereditary rings
are minimal, as we observe next.

\begin{rem}\label{hermin}
{\rm Let $R$ be a commutative hereditary ring, and let $T$ be a silting module. Denote by $C=T^+$ the corresponding cosilting module. As discussed in \cite[paragraph after Proposition 6.5]{AH}, there is  a flat epimorphism  $\lambda: R \rightarrow S$ such that $S^+\oplus \mr{Ker}\lambda^+$ is a cosilting module equivalent to $C$. Moreover, by Theorem~\ref{00},  we also have a minimal silting module $T'=S\oplus\mr{Coker}\lambda$. Now $T$ and $T'$ are both mapped to $C$ under the silting-cosilting-bijection $T\mapsto T^+$ established in \cite[Corollary 3.6]{AH1}, hence they are equivalent. This shows that $T$ is a minimal silting module.
}
\end{rem}

We now want to determine further conditions ensuring that minimality is preserved by ring extensions. Let us first explore how to relax the assumption hereditary.

\begin{lem}\label{pdim} Let $\lambda:R\to S$ be a pseudoflat ring epimorphism  such that $S_R$ has projective dimension at most one.
Then $T=S\oplus \mr{Coker}\lambda$ is a minimal silting module.
\end{lem}
\begin{proof}
We know from  \cite[Example 4.15]{AH} that $S_R$ has a presilting presentation. By the proof of \cite [Proposition 1.3]{AMV4}, the $R$-module homomorphism  $\lambda$ can be lifted to a triangle $R\stackrel {\phi}\rightarrow \sigma_{0} \rightarrow \sigma_{1}\rightarrow R[1]$
 in the derived category $D(\mr{Mod}R)$
such that $T$  is a silting module with respect to the projective presentation $\sigma=\sigma_{0}\oplus\sigma_{1}$. It follows that $\phi$ is a left Add$\sigma$-approximation. It  remains to prove that $\phi$ is left minimal. Consider a morphism  $g\in \text{End}_{D(\mr{Mod}R)}\sigma_{0}$ such that $g\phi=\phi$.
$$\xymatrix{
  R \ar[d]_{\phi} \ar[r]^{\phi} &  \sigma_{0}      \\
  \sigma_{0} \ar[ur]_{g}                     }$$
Applying the cohomology functor $\text{H}^{0}(-)$ to the diagram, we obtain the following commutative diagram.
$$\xymatrix{
  R\ar[d]_{\lambda} \ar[r]^{\lambda} &   S  \\
  S  \ar[ur]_{\text{H}^{0}(g)}                     }$$
 where $\text{H}^{0}(g)$ is an isomorphism since $\lambda$ is left minimal. Since H$^{i}(\sigma_{0})=0$ for any $i\neq 0$, we infer that $g$ is an isomorphism.  It follows that $T$ is a minimal silting $R$-module.
\end{proof}

The push-out (or coproduct) of ring epimorphisms will be an important tool for our considerations.

\begin{lem}\label{a1} {\rm \cite[Proposition 5.2]{BD},\cite[Lemma ~6.2]{chen-xi},\cite [Lemma ~4.1]{AMSTV}}
Let $R$ be an arbitrary ring, and let $\lambda:R\rightarrow S$ and $\lambda^{'}: R\rightarrow S^{'}$ be ring epimorphisms with associated bireflective subcategories $\mathcal{X} $ and $\mathcal{X}^{'}$.
Consider the push-out $S\sqcup_R S^{'}$ of $\lambda$ and $\lambda^{'}$ in the category of rings
\begin{equation}\label{po}\xymatrix{
  R \ar[d]_{\lambda^{'}} \ar[r]^{\lambda}
                & S \ar[d]^{{\mu^{'}}}  \\
  S^{'} \ar[r]_{\mu}
                &     S\sqcup_{R}S^{'}      }\end{equation}
Then $\mu$ and $\mu^{'}$ are ring epimorphisms and  the bireflective subcategory associated to the composition $\mu\lambda^{'}=\mu^{'}\lambda$ is given by $\mathcal{X}\cap\mathcal{X^{'}}$.
If $\lambda$ is pseudoflat, so is $\mu$.  If   $\lambda$ is a universal localization, so is $\mu$.
\end{lem}

\begin{lem}\label{ttt}
Let R be a commutative ring, and let $\lambda:R\rightarrow S$
and $\lambda^{'}:R\rightarrow S^{'}$ be  ring epimorphisms.
Then $\text{Coker}\lambda\otimes_RS'\cong \text{Coker}\mu$, where $\mu: S'\to S\otimes_{R}S^{'}$ is given by the push-out of $\lambda$ and $\lambda^{'}$.
\end{lem}
\begin{proof}
It is well known (see e.g.~{\rm\cite [Lemma 6.3] {chen-xi}}) that $S\sqcup_{R}S^{'}\cong S\otimes_{R}S^{'}$.
We keep the notation of diagram (\ref{po}), and
set  $\tilde{\lambda}=\mu^{'}\lambda=\mu\lambda^{'}$. We know from Lemma \ref{a1} that $\tilde{\lambda}$ is a ring epimorphism, and its corresponding bireflective subcategory  is $\tilde{\mathcal{X}}=\mathcal{X}\cap \mathcal{X}^{'}$.
Applying the functor  $-\otimes_{R}S^{'}$ to the diagram (\ref{po}), we obtain the following commutative diagram
$$\xymatrix{
  R\otimes_{R}S^{'} \ar@{.>}[dr]|-{\tilde{\lambda}\otimes_{R}S^{'}}\ar[d]_{\lambda^{'}\otimes_{R}S^{'}} \ar[r]^{\lambda\otimes_{R}S^{'}} &S\otimes_{R}S^{'}\ar[d]^{\mu^{'}\otimes_{R}S^{'}} \\
  S^{'}\otimes_{R}S^{'} \ar[r]_{\mu\otimes_{R}S^{'}} & S\otimes_{R}S^{'} \otimes_{R}S^{'}.}$$
 Notice that $\lambda'\otimes_R S'$ is an isomorphism. We claim that $\mu'\otimes_R S'$ is an isomorphism as well. To this end, we will show that both $\lambda\otimes_R S'$ and $\tilde{\lambda}\otimes_R S'$ are {$\tilde{\mc{X}}$}-reflections of $R\otimes_R S'$. Our claim will then follow from  the uniqueness of reflections.

Let us
consider the commutative diagram
$$\xymatrix{
  R\ar[d]_{\psi_{R}} \ar[r]^{\lambda} & S\ar[d]^{\psi_{S}} \\
  R\otimes_{R}S^{'} \ar[r]^{\lambda\otimes_{R}S^{'}} & S\otimes_{R}S^{'}.}$$
Recall that $\lambda$ is an $\mathcal{X}$-reflection of $R$, and $\psi_{R},~\psi_{S}$ are $\mathcal{\mathcal{X}}^{'}$-reflections of $R$ and $S$, respectively. Since
  $S\otimes_{R}S^{'}\in \tilde{\mathcal{X}}$, we conclude  that $\lambda\otimes_R S'$ is an $\tilde{\mathcal{X}}$-reflection.
  We argue similarly for $\tilde{\lambda}\otimes_R S'$, using the commutative diagram below together with the fact that $S\otimes_{R}S^{'}\otimes_{R}S^{'}\simeq S\otimes_{R}S^{'}\in \tilde{\mathcal{X}}$
$$\xymatrix{
  R\ar[d]_{\psi_{R}} \ar[r]^{\tilde{\lambda}} & S\otimes_{R}S^{'}\ar[d]^{\psi_{S\otimes_{R}{S^{'}}}} \\
  R\otimes_{R}S^{'} \ar[r]^{\tilde{\lambda}\otimes_{R}S^{'}} & S\otimes_{R}S^{'}\otimes_{R}S^{'}.}$$
Next, we observe that the  $\mc{X}'$-reflections $\psi_{S^{'}}$ and $\psi_{S\otimes_{R}S^{'}}$ are isomorphisms, because  $S^{'}$ and $S\otimes_{R}S^{'}$ are in $\mathcal{X}^{'}$.
So, we have a commutative diagram with exact rows
$$ \xymatrix{
 &R\otimes_{R}S^{'} \ar[r]^{\lambda\otimes_{R}S^{'}}\ar[d]^{\lambda^{'}\otimes_{R}S^{'}}    & S\otimes_{R}S^{'}\ar[d]^{\mu^{'}\otimes_{R}S^{'}}\ar[r]              &  \text{Coker}\lambda\otimes_{R}S^{'}\ar[d]\ar[r]  & 0  \\
&S^{'} \otimes_{R}S^{'}\ar[r]^{\mu\otimes_{R}S^{'}} \ar[d]^{\psi^{-1}_{S^{'}}}    &S\otimes_{R}S^{'}\otimes_{R}S^{'}\ar[r] \ar[d]^{\psi^{-1}_{S\otimes_{R}S^{'}}}           &  \text{Coker}\mu \otimes_{R}{S^{'}}\ar[r]\ar[d]   &0  \\
&S^{'} \ar[r]^{\mu}   &S\otimes_{R}S^{'}\ar[r]              &  \text{Coker}\mu\ar[r]   &0.
}
$$
where the isomorphisms in the first two columns induce an isomorphism $ \text{Coker}\lambda\otimes_{R}S^{'}\simeq  \text{Coker}\mu$.
\end{proof}

\begin{prop}\label{L1}
Let R be a commutative ring, and let  $T=S\oplus_{R}\mathrm{Coker}\lambda$ be a silting module arising from a ring epimorphism  $\lambda:R\rightarrow S$. Assume $\lambda^{'}:R\rightarrow S^{'}$ is a ring epimorphism such that   $ S\otimes_{R}S^{'}$ has a presilting presentation over $S^{'}$. Then $T\otimes_{R}S^{'}$ is a silting $S^{'}$-module which arises from the ring epimorphism $\mu: S'\to S\otimes_{R}S^{'}$ given by the push-out of $\lambda$ and $\lambda^{'}$.
\end{prop}
\begin{proof}
By Proposition~\ref{bijpresilting}, the ring epimorphism $\mu: S'\to S\otimes_{R}S^{'}$ induces a silting $S'$-module $S\otimes_{R}S^{'}\oplus\mr{Coker}\mu$, which by Lemma~\ref{ttt} is isomorphic to $T\otimes_R S'$.
\end{proof}

\begin{cro}\label{cases}
Let R be a commutative ring, and let  $T=S\oplus_{R}\mathrm{Coker}\lambda$ be a silting module arising from a ring epimorphism  $\lambda:R\rightarrow S$. Let $\lambda^{'}:R\rightarrow S^{'}$ be a  ring epimorphism such that the projective dimension of  $ S\otimes_{R}S^{'}$ over $S^{'}$ is at most one.

(1) If  $\lambda$ is a pseudoflat ring epimorphism, then $T\otimes_{R}S^{'}$ is a minimal silting $S^{'}$-module.

(2) If $\lambda$ is a universal localization, then $T\otimes_{R}S^{'}$  arises  from {a} universal localization.
\end{cro}
\begin{proof}
By Lemma~\ref{a1}, the ring epimorphism $\mu$ is pseudoflat. Statement (1)  then follows by combining Lemma~\ref{ttt} and~\ref{pdim}. Statement (2) follows similarly by observing that $\mu$ is a universal localization.
\end{proof}


Let us now turn to the dual situation. We will see that minimality of cosilting modules is often preserved by extensions along arbitrary ring epimorphisms.

\begin{prop}{\rm \cite[Remark 2.3]{BREAZ}}  Let  $\lambda:R\to S$ be a ring homomorphism, and
 let $C$ be a cosilting module with respect to a injective copresentation $\omega$. Then  $\mr{Hom}_R(S,C) $ is a cosilting $S$-module with respect to $\mr{Hom}_R(S,\omega)$ if and only if   $\mr{Hom}_R(S,C)$, viewed as an $R$-module, lies in $\mr{Cogen}C$.
\end{prop}

\begin{prop}\label{dualL1}
Let $R$ be a commutative ring and $C$ be a minimal cosilting module arising from a  ring epimorphism $\lambda: R\rightarrow S$.   Assume that $\lambda^{'}: R\rightarrow S^{'}$ is a  ring epimorphism such that $(S\otimes_{R}S^{'})^{+}$ has a precosilting copresentation over $S^{'}$. Then $\mathrm{Hom}_{R}(S^{'},C)$ is a minimal cosilting $S^{'}$-module.
\end{prop}
\begin{proof}
The push-out of $\lambda$ and $\lambda'$ in (\ref{po}) gives rise to a ring epimorphism $\mu:S'\to S\otimes_{R}S^{'}$ as in condition (i) of Theorem~\ref{mincos}. Hence
 $(S\otimes_{R}S^{'})^{+}\oplus \text{Ker}\mu^{+}$ is a minimal cosilting $S^{'}$-module.
On the other hand, we infer from the shape of
 $C=S^{+}\oplus\mathrm{Ker}\lambda^{+}$ that
$\text{Hom}_{R}(S^{'},C)= \text{Hom}_{R}(S^{'},S^{+})\oplus \text{Hom}_{R}(S^{'},\text{Ker}\lambda^{+})\simeq \text{Hom}_{R}(S^{'},S^{+})\oplus\text{Hom}_{R}(S^{'},(\text{Coker}\lambda)^{+})
\simeq (S\otimes_{R}S^{'})^{+}\oplus(\text{Coker}\lambda \otimes_{R}S^{'})^{+}$.
Now recall from Lemma~\ref{ttt} that $\mathrm{Coker}\lambda \otimes_{R}S^{'} \simeq \text{Coker}\mu$. We deduce that $\mathrm{Hom}_{R}(S^{'},C)\simeq(S\otimes_{R}S^{'})^{+}\oplus  \text{Ker}\mu^{+}$, which concludes the proof.
\end{proof}

\begin{cro}\label{cor1}
Let  $R$ be a commutative noetherian ring, or a commutative  ring of weak global dimension at most one. Then all minimal cosilting modules extend to minimal cosilting modules along any pseudoflat ring epimorphism.
\end{cro}
\begin{proof}
Assume that $R$ is commutative noetherian. Recall from \cite[Proposition 4.5]{AMSTV} that every  pseudoflat ring epimorphism $R\to S$ is flat. By  Remark~\ref{pf},
every minimal cosilting module $C$ arises from a flat ring epimorphism $\lambda:R \rightarrow S$. If $\lambda':R\rightarrow S'$ is a flat ring epimorphism, then $S'$ is again commutative and noetherian by  \cite[Corollary~1.2 and Proposition 1.6]{Silver}. We infer from  Lemma~\ref{a1} that the push-out of $\lambda$ and $\lambda'$ in (\ref{po}) gives rise to a flat ring epimorphism $\mu:S'\to S\otimes_{R}S^{'}$.  Then $(S\otimes_{R}S^{'})^{+}$ has a precosilting copresentation over $S^{'}$ by Remark~\ref{pf}, and the claim follows from Proposition~\ref{dualL1}.

Now assume that $R$ is a commutative  ring of weak global dimension at most one. Then every pseudoflat ring epimorphism  $R\to S$ is homological, and $S$ has  weak global dimension at most one, because the functors  $\mathrm{Tor}_{i}^{R}$ and $\mathrm{Tor}_{i}^{S}$ agree on $S$-modules for all $i\ge 1$ by \cite[Theorem 4.4]{Geigle}.
So, we can proceed as above.
By  Remark~\ref{pf},
every minimal cosilting module $C$ arises from a homological ring epimorphism $\lambda:R \rightarrow S$. The push-out of $\lambda$ with a homological ring epimorphism  $\lambda':R\rightarrow S'$ yields a homological ring epimorphism $\mu:S'\to S\otimes_{R}S^{'}$.   Then $(S\otimes_{R}S^{'})^{+}$ has a precosilting copresentation over $S^{'}$, and the claim follows from Proposition~\ref{dualL1}.
\end{proof}

Recall that a ring $R$ is said to be {\it semihereditary} if every finitely generated right ideal is projective. Moreover, a cosilting module  is said to be {\it of cofinite type} if it is equivalent to the dual $T^+$ of a silting module $T$.

\begin{cro}\label{cor2}
Over a commutative ring, every minimal cosilting module arising from  a universal localization extends to a minimal cosilting module  along any ring epimorphism.

In particular, over a commutative (semi)hereditary ring, every  cosilting module (of cofinite type) extends to a minimal cosilting module  along any ring epimorphism.
\end{cro}
\begin{proof}
Let $C$ be a minimal cosilting module arising from a universal localization $\lambda:R\to S$, and let $\lambda':R\to S'$ be an arbitrary ring epimorphism. We know from  Lemma~\ref{a1} that the push-out of $\lambda$ and $\lambda'$ in (\ref{po}) gives rise to a universal localization $\mu:S'\to S\otimes_{R}S^{'}$. Then $\mu$ is a flat ring epimorphism by \cite[Corollary 4.4]{AMSTV}. By Remark~\ref{pf}, this implies that $(S\otimes_{R}S^{'})^{+}$ has a precosilting copresentation over $S^{'}$, and the claim follows from Proposition~\ref{dualL1}.

Now assume that $R$ is semihereditary. As observed in \cite[paragraph after Proposition 6.5]{AH}, every cosilting module of cofinite type is a minimal cosilting module arising from  a universal localization. The claim then follows from the first statement. Finally, if $R$ is hereditary, all cosilting modules are of cofinite type by \cite[Theorem 3.11]{AH}. \end{proof}





\noindent {\small {{\bf Acknowledgements.}} {\small~ Lidia Angeleri H\"{u}gel  acknowledges partial support by Fondazione Cariverona, program ``Ricerca Scientifica di Eccellenza 2018",
project ``Reducing complexity in algebra, logic, combinatorics - REDCOM". Weiqing Cao acknowledges   support from China Scholarship Council (Grant No. 201906860022).}

\end{document}